\documentclass[10pt,times,letter]{article}
\usepackage[margin=1in]{geometry}

\usepackage{amsmath, amsfonts,amsthm,amssymb,bm}
\usepackage{dsfont}

\usepackage[english]{babel}

\usepackage[normalem]{ulem}

\usepackage{latexsym,amssymb,amsfonts,graphicx}
\usepackage{amsmath,dsfont}
\usepackage{verbatim}
\usepackage{mathrsfs}
\usepackage{bm}
\usepackage{color}
\usepackage{epsfig}

\usepackage{url}
\usepackage{bm}
\usepackage{natbib}

\usepackage[colorlinks,linkcolor=blue,citecolor=blue,anchorcolor=blue]{hyperref}

\newcommand{\Px}{ \mathbb{P} }
\newcommand{\Qx}{ \mathbb{Q} }
\newcommand{\N}{ \mathbb{N} }
\newcommand{\Ex}{ \mathbb{E} }
\newcommand{\ExQ}{ \mathbb{E}^Q }

\def\esssup_#1{\underset{#1}{\Xi}}
\def\essinf_#1{\underset{#1}{\mathrm{ess\,inf\, }}}
\def\argmax_#1{\underset{#1}{\mathrm{arg\,max\, }}}
\def\argmin_#1{\underset{#1}{\mathrm{arg\,min\, }}}

\newcommand{\Fx}{\mathbb{F} }

\newcommand{\too}{-\!\!\!\to}
\newcommand{\lc}{\langle}
\newcommand{\rc}{\rangle}
\newcommand{\UxFx}{\mathbb{U}^{\Px,\Fx}}
\newcommand{\WINFTY}{{\cal W}_{\Omega_{\infty},2}}

\newcommand{\F}{\mathcal{F}}

\newcommand{\R}{\mathbb{R}}

\newcommand{\Afg}{{\bf(A$_{\varepsilon,f,\rho}$)}}

\newcommand{\Anu}{{\bf(A$_{\nu}$)}}

\newcommand{\Atheta}{{\bf(A$_{\Theta}$)}}
\newcommand{\Asi}{{\bf(A$_{\sigma}$)}}

\newtheorem{theorem}{Theorem}[section]
\newtheorem{definition}{Definition}[section]

\newtheorem{proposition}[theorem]{Proposition}
\newtheorem{remark}[theorem]{Remark}
\newtheorem{lemma}[theorem]{Lemma}
\newtheorem{corollary}[theorem]{Corollary}

\allowdisplaybreaks[4]

\definecolor{Red}{rgb}{1.00, 0.00, 0.00}

\definecolor{DRed}{rgb}{0.5, 0.00, 0.00}

\definecolor{Blue}{rgb}{0.00, 0.00, 1.00}

\definecolor{Green}{rgb}{0.0, 0.4, 0.0}

\title{Large Sample Mean-Field Stochastic Optimization
\footnote{We are grateful to Nicholas Garcia Trillos, and Weinan E for valuable feedback. We would also like to thank the participants of the SIAM-FM 21 Conference, and of the 2021 SIAM Annual Meeting for valuable comments. }
}
\author{
Lijun Bo \thanks{Email: lijunbo@xidian.edu.cn, School of Mathematics and Statistics, Xidian University, Xi'an, Shaanxi Province, 710126, China} \and
Agostino Capponi \thanks{E-mail: ac3827@columbia.edu, Department of Industrial Engineering and Operations Research, Columbia University, New York, 10027, NY, USA.} \and
Huafu Liao \thanks{E-mail: huafu.liao@hu-berlin.de, Department of Mathematics, Humboldt-Universit\"at zu Berlin, Unter den Linden 6, 10099 Berlin.}}

\begin{document}
\maketitle

\begin{abstract}

We study a class of sampled stochastic optimization problems, where the underlying state process has diffusive dynamics of the mean-field type. We establish the existence of optimal relaxed controls when the sample set has finite size. The core of our paper is to prove, via $\Gamma$-convergence, that the minimizer of the finite sample relaxed problem converges to that of the limiting optimization problem.  We connect the limit of the sampled objective functional to the unique solution, in the trajectory sense, of a nonlinear Fokker-Planck-Kolmogorov (FPK) equation in a random environment. We highlight the connection between the minimizers of our optimization problems and the optimal training weights of a deep residual neural network.

\vspace{0.3 cm}


\vspace{0.3 cm}

\end{abstract}

\section{Introduction}

We investigate a class of sampled control problems, where the underlying state process has stochastic dynamics of the mean-field type, but the control process is the same for all data samples. Given an original probability space $(\Omega,\F,\Fx,\Px)$ with filtration $\Fx=(\F_t)_{t\in[0,T]}$ satisfying the usual conditions, we consider a $\theta$-controlled state process $\tilde{X}^{\theta}= (X^{\theta,1}(t),\ldots,X^{\theta,N}(t))_{t\in[0,T]}$. For $i=1,\ldots,N$, $X^{\theta,i}=(X^{\theta,i}(t))_{t\in[0,T]}$ is $d$-dimensional and has dynamics specified by
\begin{align}\label{eq:modeli}
\displaystyle dX^{\theta,i}(t) = f\left(t, \theta(t),X^{\theta,i}(t),\frac{1}{N}\sum_{j=1}^N\rho(X^{\theta,j}(t))\right)dt + \varepsilon^idW^i(t),
\end{align}
with initial condition $X^{\theta,i}(0) = X^{i}(0)$. We assume that {$f(t,\theta,x,\eta):[0,T]\times\R^m\times\R^d\times\R^{q}\to\R^d$}, $\rho:\R^d\to\R^q$, and  $W^i=(W^i(t))_{t\in[0,T]}$, $i\in\N$, are independent $p$-dimensional Brownian motions, with $\varepsilon^i\in\R^{d\times p}$. We consider a population risk minimization criterion:
\begin{align}\label{eq:control2}
\inf_{\theta\in\UxFx}J_N(\theta):=\inf_{\theta\in\UxFx}\Ex\left[L_N(\tilde{X}^{\theta}(T),\tilde{Y}(0))+\int_0^TR_N(\theta(t),\theta'(t);\tilde{X}^{\theta}(t),\tilde{Y}(0))dt\right].
\end{align}
Above, $\tilde{Y}(0) :=(Y^{1}(0),\ldots,Y^{N}(0))$ with $Y^i(0)$ being a $d$-dimensional random variable for $i=1,\ldots,N$, the notation $\theta'(t)$ denotes the first-order weak derivative w.r.t $t$ if $\theta\in{\cal H}_m^1$, and ${\cal H}_m^p$ denotes  the Sobolev space $W^{1,p}((0,T);\R^m)$ for $p\in\N$. We also use $\UxFx$ to denote the admissible strict control set under $\Px$, which will be specified in Section~\ref{sec:sampledcontrolproblem}. Further, we assume both the terminal loss function $L_N$ and the regularizer function $R_N$ to be of quadratic type.

The objective of our paper is to establish that the sequence of minimizers of the sampled objective functionals $(J_N)_{N=1}^{\infty}$ converges, as $N$ increases, to the minimizer of a limiting objective functional $J$. Because the  minimizing sequence of strict controls is {\it not} guaranteed to be precompact in $L^2(\Omega;{\cal H}_m^1)$, one cannot guarantee the convergence of the sequence of strict control minimizers to the strict control minimizer of the limiting optimization problem. We overcome this challenge by considering relaxed, rather than strict, controls and characterize a global solution to the relaxed control problem. Relaxed controls were first studied by \cite{ElKaroui87}. Using Krylov's Markovian selection theorem, \cite{HaussmannLepeltier90} established the existence of Markovian feedback policies for relaxed optimal control problems. The relaxed stochastic maximum principle was developed, respectively in singular and partially observed optimal control problems, by \cite{Bahlalietal2007} and \cite{AhmedCharalambous13}. More recently, relaxed controls have been applied to analyze the existence of Markovian equilibria and relaxed $\epsilon$-Nash equilibrium of mean field games, both in the presence of idiosyncratic and common noise; see \cite{Lacker15}, \cite{Lacker16}, \cite{CarmonaDelLa2016} and \cite{CarmonaDelarue18}.

Our {\it main result} is to show that, as $N$ gets larger, the sequence of relaxed controls which minimize the finite-dimensional sampled objective functional converges to the relaxed control minimizer of the limiting objective functional. To the best of our knowledge, this convergence result is novel in the literature, and has not been established even in the case where the underlying state dynamics are purely diffusive and have no mean-field interaction.

We next elucidate on the main technical steps and contributions. First, we prove the existence of optimal relaxed solutions to the finite sampled optimization problem stated in Eq.~\eqref{eq:relaxed-controlQ} by establishing the precompactness of the corresponding minimizing sequence (c.f. Section~\ref{sec:existence-sol-samploptim}). For any relaxed control $Q_N$ of the finite sample problem with $N$ data, we use the coordinate process consisting of Brownian drivers, input-output samples and control process $\theta_N$ to generate a sequence of empirical measure-valued processes $(\mu^N)_{N=1}^{\infty}$ (see Eq.~\eqref{eq:empiricalpmN}). We then show that this process admits a limit, which can be characterized by the joint distribution of the initial sample data, the control process and the unique solution of a nonlinear FPK equation in a random environment {(c.f.~Section~\ref{sec:limitLargeSampleJN})}. We introduce the limiting objective functional $J$, and define it in terms of the joint law of initial data, control process and state process. The main result presented in Theorem~\ref{thm:miniconver} is as follows: as $N\to\infty$, the sequence of relaxed controls which  minimize the sampled objective functional $J_N$ converges to the relaxed control which minimizes  the limiting objective functional $J$. To prove this result, we show that $J_N$ {\it $\Gamma$-converges} to $J$ on the space of relaxed controls {(c.f.~Section~\ref{sec:Gamma-convergence})}.\\

\noindent {\bf Connections to Deep Residual Neural Network (ResNet).} The class of optimization problems considered in our study can be used to characterize the optimal training weights in a deep ResNet fed with many training data. {ResNets} were first introduced in the influential study of \cite{HeRenSun2016} to remedy against the degradation problem faced by deep convolutional neural networks, where the training accuracy degrades rapidly as the network depth increases.

Leveraging the optimal control problem formulation of a deep ResNet, first proposed in \cite{E17}, and rigorously analyzed in \cite{EHanLi18}, we demonstrate how the stochastic state dynamics with mean-field interactions captures two types of {\it regularization} during the training process. The first type is {\it regularization by noise}, and consists in injecting noise into the transition function at each layer of the ResNet. This limits the amount of information carried by the units, and acts as a form of regularization because it averages over local neighborhoods of the transition function.  The second type of regularization captured by our framework is the so-called {\it batch normalization} (\cite{IoffeSzegedy2015}). Such technique was designed to remedy against the fact that the distribution of each layer's inputs changes during training, as the parameters of previous layers change. This is achieved via a normalization step that fixes the means and variances of layer inputs. Such normalization requires sampling inputs from a distribution, and leads to mean-field interactions in the dynamics of the deep {ResNet}.

There has been a growing literature on optimal control approaches to deep {ResNets}, modeled as continuous time dynamical systems. The work by \cite{ThGe2018} provide a discrete-to-continuum Gamma-convergence result for the objective function of such a stylized {ResNet}, albeit in a deterministic setup. They consider a finite training set and show that the objective function of the {ResNet} converges, as the number of layers goes to infinity, to a variational problem specified by a set of nonlinear differential equations. Unlike \cite{ThGe2018} who compare discretized multi-layer and continuous layer ResNets, the starting point of our analysis is a continuous layer ResNet. Our objective is to analyze the relation between a ResNet fed with finitely many training samples and one fed with infinitely many samples. \cite{Jabir} use a neural ODE to describe the learning process of a network. They derive the Pontryagin's optimality principle for the relaxed control problem which yields the weights of a network with finite sample size, and then provide convergence guarantees for the gradient descent algorithm. \cite{Larsson} also view deep neural networks as discretizations of controlled ODEs.

Some studies have analyzed the performance of neural networks consisting of {a single} layer. \cite{SS1} model  single-layer networks as limits of the sequence of empirical measures of interacting particles. They derive an evolution equation for the empirical measure, and update parameters through the stochastic gradient descent algorithm (SGD). \cite{Chizat} consider a neural network with a single large hidden layer. They study the performance of a non-convex particle gradient descent algorithm as the number of particles grows large, and assume a quadratic loss function to be minimized.

Other studies have considered deeper neural networks. \cite{Montanari} study two-layer neural networks, and show that the dynamics of stochastic gradient descent algorithms used for parameter learning can be approximated by a nonlinear partial differential equation (PDE). \cite{Rotskoff2018} propose an interacting particle systems model which nests multi-layer neural networks. Therein, to overcome the difficulty of minimizing the training error over the set of parameters, they study the minimum of the loss function over its empirical distribution. \cite{Du19} study a deep overparameterized {ResNet}, and prove that gradient descent achieves zero training loss. \cite{SS2} characterize multi-layer neural networks as the number of hidden layers grows large, and the number of stochastic gradient descent (SGD) iterations grows to infinity. We refer to \cite{Mallat16} for a mathematical treatment of deep neural networks, with a focus on convolutional architectures.\\

The remainder of this paper is organized as follows. In Section~\ref{trainingRNN}, we  formulate the sample-based
stochastic control problem and its corresponding relaxed form. In Section~\ref{sec:existence-sol-samploptim},  we establish the well-posedness of the sampled relaxed control problem  for finite sample size. In Section~\ref{sec:limitLargeSampleJN}, we study the convergence of the sampled objective functional under a suitable Wasserstein metric, as the number of samples grows to infinity. In Section~\ref{sec:Gamma-convergence}, we prove that the sequence of minimizers of sampled objective functionals converges to the minimizer of the limiting objective functional using  $\Gamma$-convergence. Section~\ref{sec:appresnet} shows how the proposed control framework can be used to model the training process of a deep ResNet. Section~\ref{sec:concludes} concludes. Some proofs and additional auxiliary results are delegated to an Appendix.

\section{Sampled Stochastic Control Problem and its Relaxation}\label{trainingRNN}

Consider the stochastic dynamical system given in Eq.~\eqref{eq:modeli}. The $\Fx$-adapted process $\theta=(\theta(t))_{t\in[0,T]}$ is the control strategy, which will be optimally chosen to minimize the objective criterion introduced later in this section.

We impose the following assumptions to ensure that the sampled controlled system described by \eqref{eq:modeli} is well posed:
\begin{itemize}
	\item[{\Afg}]
	\begin{enumerate}
		\item[(i)] there exists a global constant $K>0$ such that $|\varepsilon^i|\leq K$ for all $i\in\N$;
		\item[(ii)] the function $[0,T]\ni t\to f(t,0,0,0)$ is continuous;
		\item[(iii)] the function {$f(t,\theta,x,\eta):[0,T]\times\R^m\times\R^d\times\R^{q}\to\R^d$} is Lipschitz  continuous in $(\theta,x,\eta)$ uniformly in $t$, i.e., for $(\theta_1,x,\eta)$, $(\theta_2,y,\xi)\in\R^m\times\R^d\times\R^{q}$,
		\begin{align}
		|f(t,\theta_1,x,\eta)-f(t,\theta_2,y,\xi)|\leq[f]_{{\rm Lip}}[|\theta_1-\theta_2|+|x-y|+|\eta-\xi|],
		\end{align}
		where $[f]_{{\rm Lip}}$ is the Lipschitz coefficient of $f$ which is independent of the time variable $t$;
		\item[(iv)] the function $\rho:\R^d\to\R^q$ is Lipschitz continuous.
	\end{enumerate}
\end{itemize}
The Lipschitz conditions imposed in (iii)-(iv) of Assumption {\Afg} guarantee the existence of a unique strong solution of the stochastic system~\eqref{eq:modeli}.

\subsection{Sampled Control Problem}\label{sec:sampledcontrolproblem}

We begin by formulating the strict sampled optimization problem, and then introduce the corresponding relaxed version. The control process $\theta=(\theta(t))_{t\in[0,T]}$ takes values on $\Theta\subset\R^m$ in an admissible set $\UxFx$, and  {is obtained by minimizing the population risk criterion given by \eqref{eq:control2}}. We assume that
\begin{align}\label{eq:zetai}
	{\zeta^i:=(X^i(0),Y^i(0))\in\Xi_K,}
\end{align}
where {$\Xi_K:=[-K,K]^{2d}$} for a globally positive constant $K$. This means that the initial samples are assumed to have compact support. For $(\tilde{x},\tilde{y},\theta,\theta')\in(\R^{d})^N\times(\R^{d})^N\times\R^m\times\R^m$, define
\begin{align}\label{eq:LNRN}
    L_N(\tilde{x},\tilde{y}):=\frac{1}{N}\sum_{i=1}^NL(x^i,y^i),\quad R_N(\theta,\theta';\tilde{x},\tilde{y}):=\frac{1}{N}\sum_{i=1}^NR(\theta,\theta';x^i,y^i).
\end{align}
Above, $L:\R^d\times\R^d\to\R_+$ denotes the terminal loss function, and $R:\R^m\times\R^m\times\R^d\times\R^d\to\R_+$ is the regularizer of the control problem. We assume $(L,R)$ to be of the quadratic form, i.e., for $\alpha,\lambda_1,\lambda_2>0$ and $\beta\geq0$,
\begin{align}\label{eq:squareLR}
L(x,y):=\alpha|x-y|^2,\quad R(\theta,\theta';x,y):=\lambda_1|\theta|^2+\lambda_2|\theta'|^2+\beta|x-y|^2,
\end{align}
for $(x,y,\theta,\theta')\in\R^d\times\R^d\times\R^m\times\R^m$. Taking the ${\cal H}_m^1$-regularizer into account, the admissible set $\UxFx$ is defined as:
\begin{align}\label{eq:calU}
\UxFx:=\left\{\theta\in L^2(\Omega;{\cal H}_m^1);\ \theta~{\rm is\ }\Fx\mbox{-}{\rm adapted~and}~\theta\in\Theta,~{\rm a.s.~on~}(0,T)\times\Omega\right\}.
\end{align}
The admissible set $\UxFx$ is usually referred to as the {\it strict control} set under the original probability space. Throughout the paper, we assume that the state space of the parameter process $\theta$ satisfies the following property:
\begin{itemize}
\item[{\Atheta}] $\Theta\subset\R^m$ is a compact set, which is {\it not} necessarily convex.
\end{itemize}

The objective of our paper is to establish the convergence of minimizers of the sampled objective functional $J_N$ to the minimizer of a limiting objective functional $J$ as $N\to\infty$. We are confronted with two main technical challenges. First, under {\Atheta}, the problem \eqref{eq:control2} is a control problem with a non-convex policy space, hence in general a global optimum $\theta^* \in \UxFx$ is not guaranteed to exist. Specifically, the existence of an optimal solution to the optimization problem \eqref{eq:control2} does not follow from standard compactness techniques used in deterministic optimization problems. Second, even if the set $\Theta$ were convex, the convergence of the sequence of strict control minimizers to the strict control minimizer of the limiting optimization problem is not guaranteed. This is because the associated minimizing sequence of strict controls is, in general, {\it not} precompact in $L^2(\Omega;{\cal H}_m^1)$.

\subsection{Sampled Relaxed Control Problem}

We use relaxed controls to bypass the key technical challenges described at the end of the previous section. The set of relaxed controls is the space of probability measures on the data {$\zeta=(\zeta^i)_{i=1}^{\infty}$, the control process $\theta$, and the infinite-dimensional Browian motion $(W^i)_{i=1}^{\infty}$}. This is a convex set, and compact with weak topology if $\Theta$ is compact. As a result, a global solution to the relaxed control problem is always guaranteed to exist. Moreover, the minimizing sequence of relaxed controls is precompact in the space of probability measures under the weak topology. We are then able to establish the convergence of the sequence of (relaxed) minimizers of the sampled objective functional $J_N$ to the minimizer of the limiting objective functional $J$, as $N$ grows to infinity. We  provide an expression for the limiting objective functional $J$ (see Eq.~\eqref{eq:control25JQ}), and relate its representation to the solution of a FPK equation in a random environment.

We begin by establishing a canonical measurable space $(\Omega_{\infty},\F_{\infty})$ via an infinite product space such that the coordinate process $(\zeta,W,\theta)=((\zeta^i)_{i=1}^{\infty},(W^i)_{i=1}^{\infty},\theta)$ can be constructed {in the probability space $(\Omega_{\infty},\F_{\infty},Q)$, where $Q$ is referred to as a relaxed control}. We consider the  natural filtration $\Fx=(\F^{\zeta,W,\theta}_t)_{t\in[0,T]}$ generated by $(\zeta,W,\theta)$, which is the completion of the filtration flow $\sigma(\zeta)\vee\sigma(W(s),\theta(s);\ s\leq t)$ for $t\in[0,T]$. Then, the set ${\cal Q}(\nu)$ of relaxed controls is a collection of probability measures $Q$ on $(\Omega_{\infty},\F_{\infty})$ such that, under $(Q,\Fx)$, the initial sample data $\zeta$ has a given law $\nu$; $W$ is a sequence of independent Wiener processes; and $\theta\in\mathbb{U}^{Q,\Fx}$. We will formally specify ${\cal Q}(\nu)$ in Definition~\ref{def:relax-sol}. Changing the control variable $\theta\in\UxFx$ in \eqref{eq:control2} to $Q\in{\cal Q}(\nu)$, the sampled objective functional is given by
\begin{align}\label{eq:control2Q}
J_N(Q):=\Ex^Q\left[L_N(\tilde{X}^{\theta}(T),\tilde{Y}(0))+\int_0^TR_N(\theta(t),\theta'(t);\tilde{X}^{\theta}(t),\tilde{Y}(0))dt\right],
\end{align}
where, for $i=1,\ldots,N$, the state process {$(X^{\theta,i}(t))_{t\in[0,T]}$} is the strong solution to Eq.~\eqref{eq:modeli} driven by $(\zeta,W,\theta)$. Thus, the relaxed control problem in the finite sample case can be formulated as:
\begin{align}\label{eq:relaxed-controlQ}
\alpha_N:=\inf_{Q\in{\cal Q}(\nu)}J_N(Q).
\end{align}
We call $Q^*\in{\cal Q}(\nu)$ an optimal (relaxed) solution of the optimization problem \eqref{eq:relaxed-controlQ} if $J_N(Q^*)=\alpha_N$.
\begin{remark}\label{rem:sampleQ}
For a fixed sample size $N$, we only use the first $N$ components of the vector $(\zeta,W)=(\zeta^i,W^i)_{i=1}^{\infty}$ in the sampled objective functional $J_N(Q)$, where $Q\in{\cal Q}(\nu)$. As $N$ tends to infinity, the sampled objective functional $J_N(Q)$ will eventually include all components of $(\zeta,W)=(\zeta^i,W^i)_{i=1}^{\infty}$.
\end{remark}

\section{Optimal Relaxed  Solutions of Sampled Optimization Problem}\label{sec:existence-sol-samploptim}

This section establishes the existence of optimal relaxed solutions for the sampled optimization problem \eqref{eq:relaxed-controlQ}. We work with the canonical probability space, which needs to be established in terms of an infinite product space. To wit, define
\begin{align}\label{eq:OmegaNFN}
\Omega_\infty:=\Xi_K^{\mathbb{N}}\times {\cal C}_{p}^{\mathbb{N}}\times {\cal H}_m^1,\quad \F_\infty:={\cal B}(\Omega_\infty),
\end{align}
where ${\cal C}_d:=C([0,T];\R^d)$ denotes the space of continuous functions from $[0,T]$ to $\R^d$ equipped with the uniform norm $\|h\|_T:=\sup_{t\in[0,T]}|h(t)|$ for any $h\in{\cal C}_d$. For $\zeta=(\zeta^i)_{i=1}^{\infty}$ and $W=(W^i)_{i=1}^{\infty}$, we use $(\zeta,W,\theta)$ to denote the identity map on $\Omega_{\infty}$. Let $\Fx=(\F^{\zeta,W,\theta}_t)_{t\in[0,T]}$ be the complete natural filtration generated by $(\zeta,W,\theta)$. It follows from the Sobolev embedding theorem (see, e.g. \cite{Evans2010}) that $h\in{\cal H}_m^1$ if and only if $h$ equals ($dt$-a.e.) to an absolutely continuous function whose ordinary derivative (which exists $dt$-a.e.) belongs to ${\cal L}_m^2$. Here, for $p\geq1$, ${\cal L}_m^p:=L^p((0,T);\R^m)$ equipped with the $L^p$-norm $\|h\|_{{\cal L}_m^p}:=\{\int_0^T|h(t)|^pdt\}^{1/p}$ for any $h\in{\cal L}_m^p$. Then, we treat ${\cal H}_m^1$ as a subset of ${\cal C}_{m}$. We next endow the space $\Omega_{\infty}$ with the following metric: for $(\gamma,w,\vartheta)$ and $(\hat{\gamma},\hat{w},\hat{\vartheta})\in \Omega_{\infty}$, define
\begin{align}\label{eq:dinfty}
d_{\infty}((\gamma,w,\vartheta),(\hat{\gamma},\hat{w},\hat{\vartheta})):=d_{1}(\gamma,\hat{\gamma})+d_2(w,\hat{w})+d_3(\vartheta,\hat{\vartheta}).
\end{align}
The above metrics $d_i$, for $i=1,2,3$, are given by
\begin{align}\label{eq:metricd123}
d_1(\gamma,\hat{\gamma})&=\sum_{i=1}^{\infty}2^{-i}\frac{|\gamma_i-\hat{\gamma}_i|}{1+|\gamma_i-\hat{\gamma}_i|},\quad \gamma=(\gamma_i)_{i=1}^{\infty},\ \hat{\gamma}=(\hat{\gamma}_i)_{i=1}^{\infty}\in\Omega_{\infty}^0;\nonumber\\
d_2(w,\hat{w})&=\sum_{i=1}^{\infty}2^{-i}\frac{\|w_i-\hat{w}_i\|_{T}}{1+\|w_i-\hat{w}_i\|_{T}},\quad w=(w_i)_{i=1}^{\infty},\ \hat{w}=(\hat{w}_i)_{i=1}^{\infty}\in {\cal C}_p^{\N};\\
d_3(\vartheta,\hat{\vartheta})&=\big\|\vartheta-\hat{\vartheta}\big\|_{T},\quad \vartheta,\hat{\vartheta}\in{\cal H}_m^1.\nonumber
\end{align}
The space $(\Omega_{\infty},d_{\infty})$ is a Polish space (see, e.g. Section 3.8 in \cite{AliprantisBorder2006}, page 89). We next define the relaxed controls formally used in this paper. Let $\nu\in{\cal P}(\Xi_K^{\mathbb{N}})$ be a given initial sample law.
\begin{definition}\label{def:relax-sol}
	The set of relaxed controls ${\cal Q}(\nu)$ is defined to be the set of probability measures $Q$ on $(\Omega_{\infty},\F_{\infty})$ satisfying the following three properties:
	\begin{itemize}
		\item[{\rm(i)}] $Q\circ\zeta^{-1}=\nu$;
		\item[{\rm(ii)}] $W$ consists of a sequence of independent Wiener processes on $(\Omega_{\infty},\Fx,Q)$;
		\item[{\rm(iii)}] $\theta$ is an $\Fx$-adapted and $\Theta$-valued process with $\lc Q,\|\theta\|_{{\cal H}_m^1}^2\rc:=\Ex^Q[\|\theta\|_{{\cal H}_m^1}^2]<\infty$, i.e., $\theta\in\mathbb{U}^{Q,\Fx}$.
	\end{itemize}
For $Q\in{\cal Q}(\nu)$, we refer to $(\zeta,W,\theta)$ as the coordinate process corresponding to $Q$.
\end{definition}

The next proposition, whose proof is reported in the Appendix, establishes the existence of optimal relaxed solutions to the stochastic optimization problem \eqref{eq:relaxed-controlQ}.
\begin{proposition}\label{prop:reaxed-sol}
Let {\Afg} and {\Atheta} hold. Then, there exists an optimal solution to the relaxed control problem~\eqref{eq:relaxed-controlQ} in the finite sample case.
\end{proposition}

\section{Limit of Large Sampled Optimization Problem}\label{sec:limitLargeSampleJN}

In this section, we establish the convergence of the sampled objective functional $J_N$ given by \eqref{eq:control2Q} as the sample size $N$ tends to infinity. We first establish a general convergence result for a class of empirical processes which arise in our sampled controlled dynamic system. The convergence result is then used to (i) establish the limiting behavior of $J_N$ in Section~\ref{sec:limitsample}; and (ii) prove the Gamma-convergence of $J_N$ to $J$ in Section~\ref{sec:Gamma-convergence}.

\subsection{Convergence of Empirical Processes for Large Samples}\label{sec:convergencesample}

For $N\in\N$, let $Q_N\in{\cal Q}(\nu)$ where $\nu\in{\cal P}(\Xi_K^{\mathbb{N}})$ is the initial sample law. Let $(\zeta_N,W_N,\theta_N)$ be the coordinate process corresponding to $Q_N$ as in Definition~\ref{def:relax-sol}. Moreover, let $\tilde{X}_N=(X_N^{1}(t),\ldots,X_N^{N}(t))_{t\in[0,T]}$ be a solution of the following SDE:
\begin{align}\label{eq:modelinonstark}
\displaystyle dX_N^{i}(t) = f\left(t, \theta_N(t) ,X_N^{i}(t),\frac{1}{N}\sum_{j=1}^N\rho(X_N^{j}(t))\right)dt + \varepsilon^idW_N^{i}(t).
\end{align}
In other words, $X_N^i$ satisfies the SDE~\eqref{eq:modeli} driven by $(\zeta_N,W_N,\theta_N)$. Define $E:=\R^{d\times p}\times\R^d\times\R^d$ and introduce the following empirical measure-valued process given by
\begin{align}\label{eq:empiricalpmN}
\mu^N(t):=\frac{1}{N}\sum_{i=1}^N\delta_{(\xi_N^i,X_N^{i}(t))},\quad {\rm for}\ t\in[0,T].
\end{align}
Above, {for $i\geq1$, $\xi_N^i:=(\varepsilon^i,Y_N^i(0))\in\R^{d\times p}\times\R^d$}. We will show later that, for $N\geq1$, we can view $\mu^N=(\mu^N(t))_{t\in[0,T]}$ as a sequence of $\hat{S}:=C([0,T];{\cal P}_2(E))$-valued random variables, where ${\cal P}_p(E)$ is the $p$-order Wasserstein space with underlying metric space $(E,d_E)$. We also recall that {$C([0,T];{\cal P}_2(E))$ is the space of continuous ${\cal P}_2(E)$-valued functions defined on $[0,T]$}.  For $N\geq1$, we define the following joint distribution:
\begin{align}\label{eq:QxN}
\Qx^N:=Q_N\circ(\mu^N(0),\theta_N,\mu^N)^{-1}.
\end{align}
The main result of this section is to characterize the limiting behavior of the sequence of joint laws $(\Qx^N)_{N=1}^{\infty}$ (see Theorem~\ref{thm:limitP}). This is recovered as the unique solution of a FPK equation in a random environment. To start with, let ${\cal D}:=C_0^{\infty}(\R^m)$ be the space of test functions with its dual space given by ${\cal D}'$. We introduce a related parameterized operator defined on ${\cal D}$. Formally, for {$(\theta,\eta)\in{\cal C}_m\times\R^q$}, and {$(s,e)=(s,(\varepsilon,y,x))\in[0,T]\times E$}, define
{\begin{align}\label{eq:Atthetaeta}
{\cal A}^{\theta,\eta}\varphi(s,e)&:=\nabla_s\varphi(s,e)+f(s,\theta(s),x,\eta)^{\top}\nabla_x\varphi(s,e)+\frac{1}{2}{\rm tr}\left[\varepsilon\varepsilon^{\top}\nabla_{xx}^2\varphi(s,e)\right],\ \varphi\in{\cal D},
\end{align}}
where $\nabla_x\varphi:=(\frac{\partial \varphi}{\partial x_1},\ldots,\frac{\partial \varphi}{\partial x_d})^{\top}$ and $\nabla_{xx}^2\varphi:=(\frac{\partial^2\varphi}{\partial x_k\partial x_r})_{k=1,\ldots,q}^{r=1,\ldots,d}$.

\begin{theorem}\label{thm:limitP}
Let {\Afg} and {\Atheta} hold. Suppose further that, for some $\vartheta_0\in{\cal P}({\cal P}_2(E)\times{\cal C}_m)$,
\begin{align}\label{eq:QNvartheta0}
Q_N\circ(\mu^N(0),\theta_N)^{-1}\Rightarrow\vartheta_0,\quad N\to\infty.
\end{align}
Then $(\Qx^N)_{N=1}^{\infty}$ defined by \eqref{eq:QxN} converges in ${\cal P}_2({\cal P}_2(E)\times{\cal C}_m\times\hat{S})$. Moreover, if the law of a ${\cal P}_2(E)\times{\cal C}_m\times\hat{S}$-valued r.v. $(\hat{\mu}_0,\hat{\theta},\hat{\mu})$ defined on some probability space $(\hat{\Omega},\hat{\F},\hat{\Px})$ is a limit point of $(\Qx^N)_{N=1}^{\infty}$, then, $\hat{\Px}$-a.s., $\hat{\mu}$ is the unique solution to the following FPK equation in a random environment: $\hat{\mu}(0) = \hat{\mu}_0$, and for $t\in(0,T]$,
\begin{align}\label{eq:mustar}
\left\lc\hat{\mu}(t),\varphi(t)\right\rc - \left\lc\hat{\mu}(0),\varphi(0)\right\rc - \int_0^t\left\lc\hat{\mu}(s),{{\cal A}^{\hat{\theta},\lc\hat{\mu}(s),\rho\rc}}\varphi(s)\right\rc ds=0,\quad \forall~\varphi\in{\cal D}.
\end{align}
\end{theorem}

The roadmap of the proof of Theorem~\ref{thm:limitP} consists of three steps:
\begin{itemize}
\item[(i)] We prove the precompactness of the marginal distributions $({\Qx}_{\mu}^N)_{N=1}^{\infty}$ in ${\cal P}_2(\hat{S})$, where
\begin{align}\label{eq:muQxN}
{\Qx}_{\mu}^N:=Q_N\circ(\mu^N)^{-1}.
\end{align}
It thus follows from~\eqref{eq:QNvartheta0} that the sequence of measures $(\Qx^N)_{N=1}^{\infty}$ is tight.
\item[(ii)] We then prove that, for any weak limit point of a convergent subsequence of $(\Qx^N)_{N=1}^{\infty}$ of the form $\hat{\Px}\circ(\hat{\mu}_0,\hat{\theta},\hat{\mu})^{-1}$, $\hat{\mu}$ is the unique solution of a FPK equation in a random environment with initial condition $\hat{\mu}_0$, $\hat{\Px}$-a.s..
\item[(iii)] Finally, we show that $(\Qx^N)_{N=1}^{\infty}$ admits a unique weak limit point.
\end{itemize}

\noindent We next give a lemma (whose proof is reported in the Appendix), which will be used to verify the relative compactness of $(\Qx_{\mu}^N)_{N=1}^{\infty}$ in ${\cal P}(\hat{S})$ needed to prove step (i).
\begin{lemma}\label{lem:hatSequibounded-cont}
Let {\Afg} and {\Atheta} hold. Let $\epsilon>0$. Then, it holds that
\begin{align}\label{eq:QNM0}
\lim_{M\to\infty}\sup_{N\geq1}\Qx_{\mu}^N\left(\left\{\vartheta\in\hat{S};\ \sup_{t\in[0,T]}\int_E|e|^{2+\epsilon}\vartheta(t,de)> M\right\}\right)=0,
\end{align}
and for any $\varepsilon>0$,
\begin{align}\label{eq:QNM2}
\lim_{\delta\to0}\sup_{N\geq1}\Qx_{\mu}^N\left(\left\{\vartheta\in\hat{S};\ \sup_{|t-s|\leq\delta}{\cal W}_{E,2}(\vartheta(t),\vartheta(s))>\varepsilon\right\}\right)=0.
\end{align}
Above, ${\cal W}_{E,2}$ is the quadratic Wasserstein metric on ${\cal P}_2(E)$.
\end{lemma}
The following lemma completes step (i).
\begin{lemma}\label{lem:relativecomW2}
Let {\Afg} and {\Atheta} hold. Then, the sequence of marginal distributions $(\Qx_{\mu}^N)_{N=1}^{\infty}$ defined by \eqref{eq:muQxN} is relatively compact in ${\cal P}_2(\hat{S})$.
\end{lemma}

\begin{proof}
We first verify that $(\Qx_{\mu}^N)_{N=1}^{\infty}\subset{\cal P}_2(\hat{S})$. Let {$e=(\varepsilon,y,x)$ and $\hat{e}=(\hat{\varepsilon},\hat{y},\hat{x})\in E$}. We take a measure-valued process $\hat{\vartheta}\in\hat{S}$ satisfying $\sup_{t\in[0,T]}\int_{E}|\hat{e}|^2\hat{\vartheta}(t,d\hat{e})<+\infty$. We endow $\hat{S}$ with the metric
\begin{align}\label{eq:dhatS}
d_{\hat{S}}(\vartheta,\hat{\vartheta}):=\sup_{t\in[0,T]}{\cal W}_{E,2}(\vartheta(t),\hat{\vartheta}(t)),\qquad \vartheta,\hat{\vartheta}\in\hat{S}.
\end{align}
Then, for any $N\geq1$, we have that
\begin{align}\label{eq:esti2ndmomentQN}
&\int_{\tilde{S}}d_{\hat{S}}^2(\vartheta,\hat{\vartheta})\Qx_{\mu}^N(d\vartheta)=\Ex^{Q_N}\left[d_{\hat{S}}^2(\mu^N,\hat{\vartheta})\right]
=\Ex^{Q_N}\left[\sup_{t\in[0,T]}{\cal W}_{E,2}^2(\mu^N(t),\hat{\vartheta}(t))\right]\\
&\quad\leq\Ex^{Q_N}\left[\sup_{t\in[0,T]}\int_{E\times E}\left|e-\hat{e}\right|^2\mu^N(t,de)\hat{\vartheta}(t,d\hat{e})\right]\leq2\Ex^{Q_N}\left[\sup_{t\in[0,T]}\int_{E}\left|e\right|^2\mu^N(t,de)\right]
+2\sup_{t\in[0,T]}\int_{E}\left|\hat{e}\right|^2\hat{\vartheta}(t,d\hat{e}).\nonumber
\end{align}
Since $\hat{\mu}\in\hat{S}$, the 2nd term on the r.h.s. of the inequality \eqref{eq:esti2ndmomentQN} is finite. For the 1st term on the r.h.s. of the inequality \eqref{eq:esti2ndmomentQN}, using \eqref{eq:muQxN}, it follows that
{\begin{align}
&\Ex^{Q_N}\left[\sup_{t\in[0,T]}\int_{E}\left|e\right|^2\mu^N(t,de)\right]
\leq\frac{1}{N}\sum_{i=1}^N\Ex^{Q_N}\left[\sup_{t\in[0,T]}\left|(\xi_N^i,X_N^{i}(t))\right|^2\right]\nonumber\\
&\qquad\qquad\leq\frac{1}{N}\sum_{i=1}^N\Ex^{Q_N}\left[\left|\xi_N^i\right|^2\right]
+\Ex^{Q_N}\left[\sup_{t\in[0,T]}\left|\tilde{X}_N(t)\right|_N^2\right],
\end{align}
where we recall that $\tilde{X}_N(t):=(X_N^{1}(t),\ldots,X_N^{N}(t))$ for $t\in[0,T]$.} Using {\Afg}, Lemma~\ref{lem:estimateXN} and noting that $\zeta^i\in\Xi_K$ for all $i\geq1$, it follows from \eqref{eq:esti2ndmomentQN} that
{\begin{align}\label{eq:uniformdS2}
\sup_{N\geq1}\int_{{\hat{S}}}d_{\hat{S}}^2(\vartheta,\hat{\vartheta})\Qx_{\mu}^N(d\vartheta)&\leq\frac{2}{N}\sum_{i=1}^N\Ex^{Q_N}\left[\left|\xi_N^i\right|^2\right]+2\sup_{N\geq1}\Ex^{Q_N}\left[\sup_{t\in[0,T]}\left|\tilde{X}_N(t)\right|_N^2\right]
\nonumber\\
&\quad+2\sup_{t\in[0,T]}\int_{E}\left|\hat{e}\right|^2\hat{\vartheta}(t,d\hat{e})\nonumber\\
&<+\infty.
\end{align}}
This shows that $\Qx_{\mu}^N\in{\cal P}_2({\hat{S}})$ for all $N\geq1$.

We next prove that $(\Qx_{\mu}^N)_{N=1}^{\infty}\subset{\cal P}_2(\hat{S})$ is relatively compact. By Theorem 7.12 in \cite{Villani2003},  $(\Qx_{\mu}^N)_{N=1}^{\infty}$ is relatively compact in ${\cal P}_2(\hat{S})$ if and only if {\bf(I)} $(\Qx_{\mu}^N)_{N=1}^{\infty}$ is relative compact in ${\cal P}(\hat{S})$; and {\bf(II)} $(\Qx_{\mu}^N)_{N=1}^{\infty}$ satisfies the uniform integrability condition, i.e., for some $\hat{\vartheta}\in\hat{S}$,
\begin{align}\label{eq:UIab}
\lim_{R\to\infty}\sup_{N\geq1}\int_{\{\vartheta\in\hat{S};\ d_{\hat{S}}^2(\vartheta,\hat{\vartheta})\geq R\}}d_{\hat{S}}^2(\vartheta,\hat{\vartheta})\Qx_{\mu}^N(d\vartheta)=0.
\end{align}

\begin{itemize}
  \item The proof of {\bf(I)}: $(\Qx_{\mu}^N)_{N=1}^{\infty}$ is relatively compact in ${\cal P}(\hat{S})$.
\end{itemize}
By Ascoli's theorem, a subset ${\cal C}\subset\hat{S}=C([0,T];{\cal P}_2(E))$ is relatively compact if {\bf(I$_1$)}: for each $t\in[0,T]$, $\{\vartheta(t);\ \vartheta\in{\cal C}\}\subset{\cal P}_2(E)$ is relatively compact; and {\bf(I$_2$)}: ${\cal C}$ is equicontinuous under ${\cal W}_2$. Moreover, using again Theorem 7.12 in \cite{Villani2003}, {\bf(I$_1$)} holds if and only if {\bf(I$_{11}$)} holds: for each $t\in[0,T]$, $\{\vartheta(t);\ \vartheta\in{\cal C}\}$ is relatively compact in ${\cal P}(E)$; and {\bf(I$_{12}$)} holds: the following uniform integrability condition is satisfied:
\begin{align}\label{eq:UIab2}
\lim_{R\to\infty}\sup_{\vartheta\in{\cal C}}\int_{\{e\in E;\ |e|^2\geq R\}}|e|^2\vartheta(t,de)=0.
\end{align}

Let $\epsilon>0$. For $M,\delta,\varepsilon>0$, define the following subsets of $\hat{S}$:
\begin{align}
{\cal C}_1(M) &:=\left\{\vartheta\in\hat{S};\ \sup_{t\in[0,T]}\int_E|e|^{2+\epsilon}\vartheta(t,de)\leq M\right\},\
{\cal C}_2(\delta,\varepsilon):=\left\{\vartheta\in\hat{S};\ \sup_{|t-s|\leq\delta}{\cal W}_{E,2}(\vartheta(t),\vartheta(s))\leq\varepsilon\right\}.
\end{align}
Then, for any $\vartheta\in{\cal C}_1(M)$ and $t\in[0,T]$, it follows that
\begin{align}
\vartheta(t,B_R^c(0))\leq\frac{1}{R^{2+\epsilon}}\sup_{t\in[0,T]}\int_E|e|^{2+\epsilon}\vartheta(t,de)\leq \frac{M}{R^{2+\epsilon}}\to0,\quad R\to\infty,
\end{align}
where $B_R(0):=\{e\in E;\ |e|\leq R\}$ for $R>0$. On the other hand, it is easy to see that
\begin{align}
\lim_{R\to\infty}\sup_{\vartheta\in{\cal C}_1(M)}\int_{\{e\in E;\ |e|^2\geq R\}}|e|^2\vartheta(t,de)\leq\lim_{R\to\infty}\frac{1}{R^{\epsilon/2}}\sup_{\vartheta\in{\cal C}_1(M)}\int_{E}|e|^{2+\epsilon}\vartheta(t,de)\leq \lim_{R\to\infty}\frac{M}{R^{\epsilon/2}}=0.
\end{align}
By {\bf(I$_{11}$)} and {\bf(I$_{12}$)}, this implies that ${\cal C}_1(M)\subset\hat{S}$ satisfies {\bf(I$_{1}$)}. To continue, fix $\varepsilon>0$, by \eqref{eq:QNM0} of Lemma~\ref{lem:hatSequibounded-cont} there exists $N_0=N_0(\varepsilon)\geq1$ and $M_0=M_0(\varepsilon)\geq1$ such that $\sup_{N\geq N_0}\Qx_{\mu}^N({\cal C}_1^c(M_0))\leq\frac{\varepsilon}{2}$. Note that $\lim_{M\to\infty}\Qx_{\mu}^N({\cal C}_1^c(M))=0$ for any $N\geq1$ by \eqref{eq:QNM0}. It follows that there exists $M_1=M_1(\varepsilon)$ large enough such that $\sup_{1\leq N\leq N_0}\Qx_{\mu}^N({\cal C}_1^c(M_1))\leq\frac{\varepsilon}{2N_0}$. Then, let $M:=M_0\vee M_1$, and hence  $\sup_{N\geq1}\Qx_{\mu}^N({\cal C}_1^c(M))\leq\varepsilon$. By applying the limiting result~\eqref{eq:QNM2} of Lemma~\ref{lem:hatSequibounded-cont}, $\lim_{\delta\to0}\sup_{N\geq1}\Qx_{\mu}^N({\cal C}_2(\delta,n^{-1}))=0$ for each $n\geq1$. Then, there exists $\delta_n>0$ satisfying $\lim_{n\to\infty}\delta_n=0$ such that $\sup_{N\geq1}\Qx_{\mu}^N({\cal C}_2^c(\delta_n,n^{-1}))\leq\frac{\varepsilon}{2^n}$. Define ${\cal C}:={\cal C}_1(M)\cap(\bigcap_{n\geq1}{\cal C}_2(\delta_n,n^{-1}))\subset\hat{S}$. Then ${\cal C}$ is relatively compact in $\hat{S}$, and it follows from the above given estimates that $\sup_{N\geq1}\Qx_{\mu}^N({\cal C}^c)\leq2\varepsilon$. This shows that $(\Qx_{\mu}^N)_{N=1}^{\infty}$ is relatively compact in ${\cal P}(\hat{S})$.
\begin{itemize}
  \item The proof of {\bf(II)}: $(\Qx_{\mu}^N)_{N=1}^{\infty}$ satisfies the uniform integrability condition \eqref{eq:UIab}.
\end{itemize}
For any $N\geq1$, it holds that
\begin{align}
&\int_{{\hat{S}}}d_{\hat{S}}^{2+\epsilon}(\vartheta,\hat{\vartheta})\Qx_{\mu}^N(d\vartheta)\leq\Ex^{Q_N}\left[\sup_{t\in[0,T]}\left(\int_{E\times E}\left|e-\hat{e}\right|^2\mu^N(t,de)\hat{\vartheta}(t,d\hat{e})\right)^{\frac{2+\epsilon}{2}}\right]\nonumber\\
&\qquad\leq2^{1+\epsilon}\Ex^{Q_N}\left[\sup_{t\in[0,T]}\left(\int_{E}\left|e\right|^2\mu^N(t,de)\right)^{\frac{2+\epsilon}{2}}\right]
+2^{1+\epsilon}\left(\sup_{t\in[0,T]}\int_{E}\left|\hat{e}\right|^2\hat{\vartheta}(t,d\hat{e})\right)^{\frac{2+\epsilon}{2}}.
\label{eq:d2plusQNint}
\end{align}
It follows from Jensen's inequality that, for some constant $C_{\epsilon}>0$ which only depends on $\epsilon$,
\begin{align}
&\sup_{N\geq1}\Ex^{Q_N}\left[\sup_{t\in[0,T]}\left(\int_{E}\left|e\right|^2\mu^N(t,de)\right)^{\frac{2+\epsilon}{2}}\right]\leq
\sup_{N\geq1}\frac{1}{N}\sum_{i=1}^N\Ex^{Q_N}\left[\sup_{t\in[0,T]}\left|(\xi_N^i,X_N^{i}(t))\right|^{2+\epsilon}\right]\nonumber\\
&\qquad
\leq C_{\epsilon}\left\{\sup_{N\geq1}\frac{1}{N}\sum_{i=1}^N\Ex^{Q_N}\left[\left|\xi_N^i\right|^{2+\epsilon}\right]
+\sup_{N\geq1}\Ex^{Q_N}\left[\sup_{t\in[0,T]}\left|\tilde{X}_N(t)\right|_N^{2+\epsilon}\right]
\right\}.
\end{align}
Observe that $\hat{\vartheta}\in\hat{S}$. Using the assumption {\Afg}, Lemma~\ref{lem:estimateXN} and noting that $\zeta^i\in\Xi_K$ for all $i\geq1$, it follows from \eqref{eq:esti2ndmomentQN} and \eqref{eq:d2plusQNint} that, as $R\to\infty$,
\begin{align}
\sup_{N\geq1}\int_{\{\vartheta\in {\hat{S}};\ d_{\hat{S}}^2(\vartheta,\hat{\vartheta})\geq R\}}d_{\hat{S}}^2(\vartheta,\hat{\vartheta})\Qx_{\mu}^N(d\vartheta)\leq
\frac{1}{R^{{\epsilon}/{2}}}\sup_{N\geq1}\int_{{\hat{S}}}d_{\hat{S}}^{2+\epsilon}(\vartheta,\hat{\vartheta})\Qx_{\mu}^N(d\vartheta)\to0,
\end{align}
i.e., the uniform integrability condition \eqref{eq:UIab} holds. This completes the proof of the lemma.
\end{proof}

The following proposition completes step (ii).
\begin{proposition}\label{prop:limitP1}
Under the assumptions of Theorem \ref{thm:limitP} hold, $(\Qx^N)_{N=1}^{\infty}\subset{\cal P}({\cal P}_2(E)\times{\cal C}_m\times\hat{S})$ is tight. If the law of a ${\cal P}_2(E)\times{\cal C}_m\times\hat{S}$-valued r.v. $(\hat{\mu}_0,\hat{\theta},\hat{\mu})$ defined on some probability space $(\hat{\Omega},\hat{\F},\hat{\Px})$ is the weak limit of a convergent subsequence of $(\Qx^N)_{N=1}^{\infty}$, then, $\hat{\Px}$-a.s. $\hat{\mu}$ is the unique solution of FPK equation \eqref{eq:mustar} with initial condition $\hat{\mu}(0)=\hat{\mu}_0$.
\end{proposition}

\begin{proof} The tightness of $(\Qx^N)_{N=1}^{\infty}$ follows from the condition \eqref{eq:QNvartheta0} and Lemma~\ref{lem:relativecomW2}.
Since $\hat{\Px}\circ(\hat{\mu}_0,\hat{\theta},\hat{\mu})^{-1}$ is the weak limit of a convergent subsequence of $(\Qx^N)_{N=1}^{\infty}$, Skorokhod representation theorem implies the existence of a probability space $(\Omega^*,\F^*,\Px^*)$, a sequence of ${\cal P}_2(E)\times{\cal C}_m\times\hat{S}$-valued r.v.s $(\mu_0^{N,*},\theta_N^{*},\mu^{N,*})$ and $(\mu^{*}_0,\theta^{*},\mu^{*})$ satisfying $\Px^*\circ(\mu_0^{N,*},\theta_N^{*},\mu^{N,*})^{-1}= \Qx^N$, $\Px^*\circ(\mu^{*}_0,\theta^{*},\mu^{*})^{-1}=\hat{\Px}\circ(\hat{\mu}(0),\hat{\theta},\hat{\mu})^{-1}$, and $\Px^*$-a.s., as $N\to\infty$,
\begin{align}\label{eq:asNinfty}
\mu_0^{N,*}\Rightarrow\mu_0^{*}\ {\rm in}\ {\cal P}_2(E);\quad \theta_N^{*}\to\theta^*\ {\rm in}\ {\cal C}_m;\quad \mu^{N,*}\to\mu^*,\ {\rm in}\ (\hat{S},d_{\hat{S}}).
\end{align}
Moreover, we also have that, for all $p\geq1$,
\begin{align}\label{eq:Estarcon}
&\lim_{N\to\infty}\Ex^{*}\left[\sup_{t\in[0,T]}\left|\lc\mu^{N,*}(t),\varphi(t)\rc-\lc\mu^{N,*}(0),\varphi(0)\rc-\int_0^t\lc\mu^{N,*}(s),{{\cal A}^{\theta_N^{*},\lc\mu^{N,*}(s),\rho\rc}}\varphi(s)\rc ds\right|^{2p}\right]\\
&=\lim_{N\to\infty}\Ex^{Q_N}\left[\sup_{t\in[0,T]}\left|\lc\mu^N(t),\varphi(t)\rc-\lc\mu^N(0),\varphi(0)\rc-\int_0^t\lc\mu^N(s),{{\cal A}^{\theta_N,\lc\mu^N(s),\rho\rc}}\varphi(s)\rc ds\right|^{2p}\right]=0.\nonumber
\end{align}
The first equality in \eqref{eq:Estarcon} follows from the fact that $\Px^*\circ(\mu_0^{N,*},\theta_N^{*},\mu^{N,*})^{-1}= \Qx^N$ with $\Qx^N:=Q_N\circ(\mu^N(0),\theta_N,\mu^N)^{-1}$. As for the second equality in \eqref{eq:Estarcon}, observe that an application of It\^o's formula yields ${\cal M}^N(t)=\lc\mu^N(t),\varphi(t)\rc-\lc\mu^N(0),\varphi(0)\rc-\int_0^t\lc\mu^N(s),{{\cal A}^{\theta_N,\lc\mu^N(s),\rho\rc}}\varphi(s)\rc ds$  for any $t\in[0,T]$, where ${\cal M}^N(t):=\frac{1}{N}\sum_{i=1}^N\int_0^t\nabla_x\varphi(s,\xi^i,Z_N^i(s),X_N^i(s))^{\top}\varepsilon^idW_N^i(s)$. Observe that the test function $\varphi\in{\cal D}$, and $(W_N^1,\ldots,W_N^N)$ are independent Wiener processes under the probability measure $Q_N$. Then, from the BDG inequality, there exists a constant $C>0$, independent of $N$, such that
\begin{align*}
\varlimsup_{N\to\infty}\Ex^{Q_N}\left[\sup_{t\in[0,T]}\left|{\cal M}^N(t)\right|^{2p}\right]\leq\lim_{N\to\infty}\frac{C}{N}=0.
\end{align*}
This shows that the second equality in \eqref{eq:Estarcon} must hold.

We next claim that for any $t\in[0,T]$ and test function $\varphi\in{\cal D}$, $\Px^*$-a.s.
\begin{align}\label{eq:operatorconvergece0}
\lim_{N\to\infty}\Upsilon_{t,\varphi}(\mu^{N,*},\theta_N^*,\lc\mu^{N,*},\rho\rc)=\Upsilon_{t,\varphi}(\mu^{*},\theta^*,\lc\mu^{*},\rho\rc),
\end{align}
where the mapping $\Upsilon_{t,\varphi}:\hat{S}\times{\cal C}_m\times{\cal C}_1\to\R$ is defined as:
\begin{align}\label{eq:Gammamaping}
\Upsilon_{t,\varphi}(\mu,\theta,h):=\int_0^{t} \lc\mu(s),{{\cal A}^{\theta,h(s)}}\varphi(s)\rc ds.
\end{align}
Here, the definition of {${\cal A}^{\theta,h(s)}$} is given in \eqref{eq:Atthetaeta}. Then, for any $(\theta_i,h_i)\in{\cal C}_m\times{\cal C}_1$ with $i=1,2$ and $\mu\in\hat{S}$, it follows from {\Afg} that there exists a constant $C_{\varphi}>0$ such that
\begin{align}\label{eq:LipGamma}
&\left|\Upsilon_{t,\varphi}(\mu,\theta_1,h_1)-\Upsilon_{t,\varphi}(\mu,\theta_2,h_2)\right|
\leq C_{\varphi}\left[\left\|\theta_1-\theta_2\right\|_T+\left\|h_1-h_2\right\|_T\right].
\end{align}
Moreover, by {\Afg}, $|{{\cal A}^{\theta,h(s)}}\varphi(s,e)|\leq C_{\varphi}[1+\|\theta\|_T+\|h\|_T+|e|^2]$. Hence, for any $s\in[0,T]$, by \eqref{eq:asNinfty} and Theorem~7.12 of \cite{Villani2003}, we arrive at the conclusion that, $\Px^*$-a.s., $\lc\mu^{N,*}(s),{{\cal A}^{\theta,h(s)}}\varphi(s)\rc\to\lc\mu^{*}(s),{{\cal A}^{\theta,h(s)}}\varphi(s)\rc$ as $N\to\infty$. Further, for $q(e):=|e|^2$, it holds that, $\Px^*$-a.s.
\begin{align}
\sup_{N\geq1}\left|\left\lc\mu^{N,*}(s),{{\cal A}^{\theta,h(s)}}\varphi(s)\right\rc\right|&\leq C_{\varphi,h}+C_{\varphi,h}\sup_{N\geq1}\left\lc\mu^{N,*}(s),q\right\rc=C_{\varphi,h}+C_{\varphi,h}\sup_{N\geq1}{\cal W}_{{E},2}(\mu^{N,*}(s),\delta_0)\nonumber\\
&\leq C_{\varphi,h} +C_{\varphi,h}{\cal W}_{{ E},2}(\mu^{*}(s),\delta_0)+C_{\varphi,h}\sup_{N\geq1}d_{\hat{S}}(\mu^{N,*},\mu^{*}),
\end{align}
for some positive constant $C_{\varphi,h}$ which is independent of $N$. Using the limiting results given in~\eqref{eq:asNinfty}, we then obtain that $d_{\hat{S}}(\mu^{N,*},\mu^{*})\to0$ as $N\to\infty$, $\Px^*$-a.s. Note that $\mu^*\in\hat{S}$. Then, it follows from the dominated convergence theorem that, for $t\in[0,T]$,
\begin{align}\label{eq:dct00}
\lim_{N\to\infty}\int_0^t\lc\mu^{N,*}(s),{{\cal A}^{\theta,h(s)}}\varphi(s)\rc ds=\int_0^t\lc\mu^{*}(s),{{\cal A}^{\theta,h(s)}}\varphi(s)\rc ds,\quad \Px^*\mbox{-a.s.}
\end{align}
Using \eqref{eq:LipGamma}, \eqref{eq:asNinfty} and  \eqref{eq:dct00}, we deduce that, $\Px^*$-a.s.
\begin{align}
&\left|\Upsilon_{t,\varphi}(\mu^{N,*},\theta_N^*,\lc\mu^{N,*},\rho\rc)-\Upsilon_{t,\varphi}(\mu^{*},\theta^*,\lc\mu^{*},\rho\rc)\right|
\leq\left|\Upsilon_{t,\varphi}(\mu^{N,*},\theta_N^*,\lc\mu^{N,*},\rho\rc)-\Upsilon_{t,\varphi}(\mu^{N,*},\theta^*,\lc\mu^{*},\rho\rc)\right|\nonumber\\
&\quad+\left|\Upsilon_{t,\varphi}(\mu^{N,*},\theta^*,\lc\mu^{*},\rho\rc)-\Upsilon_{t,\varphi}(\mu^{*},\theta^*,\lc\mu^{*},\rho\rc)\right|\nonumber\\
&\quad\leq C_{\varphi}\{\|\theta_N^*-\theta^*\|_T + \|\lc\mu^{N,*},\rho\rc-\lc\mu^{*},\rho\rc\|_T\}+\left|\Upsilon_{t,\varphi}(\mu^{N,*},\theta^*,\lc\mu^{*},\rho\rc)-\Upsilon_{t,\varphi}(\mu^{*},\theta^*,\lc\mu^{*},\rho\rc)\right|\nonumber\\
&\quad\to0,\quad N\to\infty.
\end{align}
This proves the limit in~\eqref{eq:operatorconvergece0}. By applying Fatou's lemma, \eqref{eq:Estarcon} and \eqref{eq:operatorconvergece0}, we obtain that
\begin{align}\label{eq:Exstar0}
&\quad\hat{\Ex}\left[\sup_{t\in[0,T]}\left|\lc\hat{\mu}(t),\varphi(t)\rc-\lc\hat{\mu}(0),\varphi(0)\rc-\int_0^t\lc\hat{\mu}(s),{{\cal A}^{\hat{\theta},\lc\hat{\mu}(s),\rho\rc}}\varphi(s)\rc ds\right|^{2p}\right]\notag\\
&=\Ex^{*}\left[\sup_{t\in[0,T]}\left|\lc\mu^{*}(t),\varphi(t)\rc-\lc\mu^{*}(0),\varphi(0)\rc-\int_0^t\lc\mu^{*}(s),{{\cal A}^{\theta^{*},\lc\mu^{*}(s),\rho\rc}}\varphi(s)\rc ds\right|^{2p}\right]\\
& \leq\liminf_{N\to\infty}\Ex^{*}\left[\sup_{t\in[0,T]}\left|\lc\mu^{N,*}(t),\varphi(t)\rc-\lc\mu^{N,*}(0),\varphi(0)\rc-\int_0^t\lc\mu^{N,*}(s),{{\cal A}^{\theta_N^{*},\lc\mu^{N,*}(s),\rho\rc}}\varphi(s)\rc ds\right|^{2p}\right]=0.\nonumber
\end{align}
This proves~\eqref{eq:mustar} for all $\omega\in\hat{\Omega}_0$ with some $\hat{\Omega}_0\subset\hat{\Omega}$ satisfying $\hat{\Px}(\hat{\Omega}_0)=1$.

We next prove the uniqueness of a solution to the FPK equation \eqref{eq:mustar} in the trajectory sense. This is done by verifying the conditions (DH1)-(DH4) imposed in Theorem 4.4 of \cite{Manitaetal2015}. For fixed $\omega\in\hat{\Omega}_0$, and for $(t,e,\mu)\in[0,T]\times E\times{\cal P}_2(E)$,  define
\begin{align}
\begin{cases}
\displaystyle A(e):=\frac{1}{2}\left(
                           \begin{array}{c}
                             0\\
                             \varepsilon \\
                           \end{array}
                         \right)(0,\varepsilon^{\top}),\quad b_{\omega}(t,e,\mu):=\left(
                                                                                                  \begin{array}{c}
                                                                                                    0 \\
                                                                                                    f(t,\theta^*(t,\omega),x,\lc\mu,\rho\rc) \\
                                                                                                  \end{array}
                                                                                                \right);\label{eq:ABomega}\\[1.2em]
\displaystyle L_{\omega}(\mu):={\rm tr}[A(e)\nabla_{ee}^2] + b_{\omega}(t,e,\mu)^{\top}\nabla_e.
\end{cases}
\end{align}
It then follows from~\eqref{eq:ABomega} that $\sqrt{A(e)}$ is twice differentiable in $e$, and hence the assumption (DH1) in Theorem 4.4 of \cite{Manitaetal2015} is satisfied. Choose the convex function $\Phi\in C^2(E)$ as $\Phi(e)=1+|e|^2$ for $e\in E$. For any {$\hat{e}=(\hat{\varepsilon},\hat{y},\hat{x})\in E$}, it follows from {\Afg} that there exists a constant $C_{f,\phi}>0$ (which may vary from line to line) such that
\begin{align}
&(b_{\omega}(t,e+\hat{e},\mu)-b_{\omega}(t,e,\mu))^{\top}\hat{e}=\left(f(t,\theta^*(t,\omega),x+\hat{x},\lc\mu,\rho\rc)-f(t,\theta^*(t,\omega),x,\lc\mu,\rho\rc)\right)^{\top}\hat{x}\nonumber\\
&\qquad\leq C_{f,\phi}|\hat{x}|^2\leq C_{f,\phi}\Phi(\hat{e}),
\end{align}
and there exists a constant $C_{f,\phi,\mu,\omega}>0$ such that
\begin{align}
L_{\omega}(\mu)\Phi(e)={\rm tr}[A(e)\nabla_{ee}^2\Phi(e)] + b_{\omega}(t,e,\mu)^{\top}\nabla_e\Phi(e)\leq C_{f,\phi,\mu,\omega}\Phi(e).
\end{align}
It also follows from {\Afg}  that
\begin{align}
&\frac{|f(t,\theta^*(t,\omega),x,\lc\mu,\rho\rc)|^2}{\Phi(e)}+\frac{|\varepsilon|^4}{\Phi^2(e)}\leq \frac{1+\|\theta^*(\omega)\|_T^2+|e|^2+|\lc\mu,\rho\rc|^2}{\Phi(e)}+\frac{|e|^4}{\Phi^2(e)}\leq 1+\|\theta^*(\omega)\|_T^2++|\lc\mu,\rho\rc|^2.
\end{align}
This complete the verification of the condition (DH2). For any $\mu,\nu\in{\cal P}_2(E)$, it follows from \eqref{eq:ABomega} and {\Afg} that
\begin{align}\label{eq:bomwssdf}
&\left|b_{\omega}(t,e,\mu)-b_{\omega}(t,e,\nu)\right|=\left|f(t,\theta^*(t,\omega),x,\lc\mu,\rho\rc)-f(t,\theta^*(t,\omega),x,\lc\nu,\rho\rc)\right|\\
&\quad\leq [f]_{\rm Lip}|\lc\mu,\rho\rc-\lc\nu,\rho\rc|=[f]_{\rm Lip}[\rho]_{\rm Lip}\left|\left\lc\mu-\nu,\frac{\rho}{[\rho]_{\rm Lip}}\right\rc\right|\leq[f]_{\rm Lip}[\rho]_{\rm Lip}{\cal W}_{E,2}(\mu,\nu)=:G({\cal W}_{E,2}(\mu,\nu)),\nonumber
\end{align}
where $G(y):=[f]_{\rm Lip}[\rho]_{\rm Lip}y$ for $y\in[0,\infty)$. Obviously, the function $G$ is continuous and increasing on $[0,\infty)$ with $G(0)=0$. This verifies the condition (DH3). Next, we verify the condition (DH4). Take $\Psi(e)=\sqrt{\Phi(e)}$ for $e\in E$. Hence,  $\Psi\in C^2(E)$, $\Psi\geq1$, and $|\nabla_e\Psi(e)|+|\nabla_{ee}^2\Psi(e)|\leq C$ for some constant $C>0$. Moreover, it holds that $\lim_{|e|\to\infty}\Psi(e)=+\infty$. We deduce from \eqref{eq:ABomega} that, for any $\mu\in{\cal P}_2(E)$,
\begin{align}
\frac{\left|A(e)\nabla_e\Psi(e)\right|\sqrt{\Phi(e)}}{\Psi(e)}+\frac{\left|\sqrt{A(e)}\nabla_e\Psi(e)\right|^2}{\Psi^2(e)}+\frac{|L_{\omega}(\mu)\Psi(e)|}{\Psi(e)}
\leq C_{f,\phi,\mu,\omega}\Phi(e).
\end{align}
For fixed $\omega\in\hat{\Omega}_0$, the uniqueness of a solution to \eqref{eq:mustar} follows from Theorem 4.4 in \cite{Manitaetal2015}.
\end{proof}

The next lemma concludes the step (iii) outlined in the proof roadmap. Its proof follows directly from the Gluing lemma (see, e.g. Lemma 7.6 in \cite{Villani2003}) and is omitted here.
\begin{lemma}\label{lem:unilimit}
Let assumptions of Theorem \ref{thm:limitP} hold. Then, the precompact sequence $(\Qx^N)_{N=1}^{\infty}$ has a unique weak limit point.
\end{lemma}

We now have all the ingredients to prove the main result (Theorem~\ref{thm:limitP}) of this section.
\begin{proof}[{\bf Proof of Theorem~\ref{thm:limitP}}]
It follows from Proposition \ref{prop:limitP1} and Lemma \ref{lem:unilimit} that $(\Qx^N)_{N=1}^{\infty}\subset{\cal P}({\cal P}_2(E)\times{\cal C}_m\times\hat{S})$ is convergent under the weak topology. Let us endow $O:={\cal P}_2(E)\times{\cal C}_m\times\hat{S}$ with the following metric: for $o_i=(\vartheta_{0i},w_i,\vartheta_i)\in O$ with $i=1,2$,
\begin{align}\label{eq:dO}
d_{O}(o_1,o_2):={\cal W}_{E,2}(\vartheta_{01},\vartheta_{02})+\|w_1-w_2\|_T + d_{\hat{S}}(\vartheta_1,\vartheta_2).
\end{align}
Then, using assumptions {\Afg} and {\Atheta}, the fact that $\zeta_N=(\zeta_N^i)_{i=1}^{\infty}\in\Xi_K^{\N}$, and Lemma~\ref{lem:estimateXN}, for $\hat{o}=(\delta_{0},0,\delta_0)\in O$ and $\epsilon>0$, it follows that
\begin{align}
&\sup_{N\geq1}\int_{\{o\in O;\ d_{O}(o,\hat{o})\geq R\}}d_{O}^{2+\epsilon}(o,\hat{o})\Qx^N(do)\leq C\sup_{N\geq1}\Ex^{Q_N}\left[\left|{\cal W}_{E,2}^2(\mu^N(0),\delta_{0})+\|\theta_N\|_T^2
+d_{\hat{S}}^2(\mu^N,\delta_0)\right|^{\frac{2+\epsilon}{2}}\right]\nonumber\\
&\qquad\qquad\leq C \sup_{N\geq1}\Ex^{Q_N}\left[\left|\frac{1}{N}\sum_{i=1}^N(|\varepsilon^i|^2+|\zeta_N^i|^2)+\|\theta_N\|_T^2
+\frac{1}{N}\sum_{i=1}^N\|X_N^i\|_T^2\right|^{\frac{2+\epsilon}{2}}\right]\leq C_{K,\epsilon,T}<+\infty,
\end{align}
where $C$ and $C_{K,\epsilon,T}$ are some positive constants independent of $N$. This implies that
\begin{align}\label{eq:uniforminteoo}
\lim_{R\to\infty}\sup_{N\geq1}\int_{\{\vartheta\in O;\ d_{O}(o,\hat{o})\geq R\}}d_O^2(o,\hat{o})\Qx^N(do)=0.
\end{align}
Then, the convergence of $(\Qx^N)_{N=1}^{\infty}$ in ${\cal P}_2(O)$ follows from Theorem 7.12 in \cite{Villani2003} along with the uniform integrability result given in~\eqref{eq:uniforminteoo}.
\end{proof}

Lastly, we show the uniqueness of the weak limit point of the marginal distributions $(\Qx_\mu^N)_{N=1}^{\infty}$ defined by \eqref{eq:muQxN}.
Lemma~\ref{lem:relativecomW2} shows that $(\Qx_{\mu}^N)_{N=1}^{\infty}\subset{\cal P}_2(\hat{S})$ is precompact. The following corollary is an immediate consequence of Theorem \ref{thm:limitP}.
\begin{corollary}\label{coro:uniquemuN}
Let assumptions of Theorem \ref{thm:limitP} hold. Then, the precompact sequence $(\Qx_{\mu}^N)_{N=1}^{\infty}\subset{\cal P}_2(\hat{S})$ has a unique limit $\Qx_{\mu}^*\in{\cal P}_2(\hat{S})$ satisfying ${\cal W}_{\hat{S},2}(\Qx_{\mu}^N,\Qx_{\mu}^*)\to0$ as $N\to\infty$.
\end{corollary}

\subsection{A Sufficient Condition for Weak Convergence~\eqref{eq:QNvartheta0}}\label{sec:assumptionAnu}

This section provides an easily verifiable sufficient condition on the initial sample law $\nu\in{\cal P}(\Xi_K^{\N})$ that guarantees the weak convergence \eqref{eq:QNvartheta0} assumed in Theorem~\ref{thm:limitP}:
\begin{itemize}
\item[{\Anu}] For $N\in\N$, define the mapping $I_N:\Xi_K^{\N}\to{\cal P}_2(E)$ as follows: for any $\hat{\zeta}=(\hat{X}^{i},\hat{Y}^{i})_{i=1}^{\infty}\in\Xi_K^{\N}$, $I_N(\hat{\zeta}):=\frac{1}{N}\sum_{i=1}^N\delta_{(\varepsilon^i,\hat{Y}^i,\hat{X}^{i})}$.
	Then, there exists a measurable mapping $I_*:\Xi_K^{\N}\to{\cal P}_2(E)$ such that
	\begin{align}\label{eq:assumptionAnu3}
	\nu\left(\left\{\hat{\zeta}\in\Xi_K^{\N};\ \lim_{N\to\infty}{\cal W}_{E,2}(I_N(\hat{\zeta}),I_*(\hat{\zeta}))=0\right\}\right)=1.
	\end{align}
\end{itemize}

The following remark presents an example of initial laws of training samples that satisfy {\Anu}:
\begin{remark}\label{rem:examAnu1}
	Consider any sequence of i.i.d. $\Xi_K$-valued r.v.s {$(\hat{X}^i,\hat{Y}^i)_{i=1}^{\infty}$} on some probability space $(\hat{\Omega},\hat{\F},\hat{\Px})$. Set $\hat{\zeta}=(\hat{\zeta}^i)_{i=1}^{\infty}=(\hat{X}^i,\hat{Y}^i)_{i=1}^{\infty}$ and hence $\hat{\zeta}\in\Xi_K^{\N}$. Then, for any sequence {$(\varepsilon^i)_{i=1}^{\infty}$ satisfying $\lim_{i\to\infty}\varepsilon^i=\varepsilon^*$,} the law of large of number (LLN) yields $\hat{\Px}(\{\omega\in\hat{\Omega};\ \lim_{N\to\infty}{\cal W}_{E,2}(I_N(\hat{\zeta}(\omega)),I_*)=0\})=1$.
In this specific setup,  $I_*:=\delta_{(\varepsilon^*,\gamma^*,\sigma^*)}\otimes\hat{\Px}\circ(\hat{\zeta}^1)^{-1}$. Consider the initial sample law $\nu:=\hat{\Px}\circ(\hat{\zeta})^{-1}$, then it holds that
	\begin{align}
	\nu\left(\left\{\hat{\zeta}\in\Xi_K^{\N};\ \lim_{N\to\infty}{\cal W}_{E,2}(I_N(\hat{\zeta}),I_*)=0\right\}\right)=\hat{\Px}\left(\left\{\omega\in\hat{\Omega};\ \lim_{N\to\infty}{\cal W}_{E,2}(I_N(\hat{\zeta}(\omega)),I_*)=0\right\}\right)=1.
	\end{align}
	Hence, the assumption {\Anu} is satisfied in this specific setup.
\end{remark}

Note that the convergence of relaxed controls $(Q_N)_{N=1}^{\infty}$ does not imply the weak convergence \eqref{eq:QNvartheta0}. The reason is that the empirical distribution of the initial data may not converge. Condition {\Anu} guarantees that the distribution of initial data is well behaved.
The following lemma, proven in the Appendix, shows that if the convergence of $(Q_N)_{N=1}^{\infty}$ converges, then {\Anu} implies the weak convergence in~\eqref{eq:QNvartheta0}.
\begin{lemma}\label{lem:Wtouni}
Let {\Anu} hold. Consider an arbitrary sequence $(Q_N)_{N=1}^{\infty},Q\subset{\cal Q}(\nu)$ with $\lim_{N\to\infty}{\cal W}_{{\Omega_{\infty}},2}(Q_N,Q)=0$, then $Q_N\circ(\zeta_N,I_N(\zeta_N),\theta_N)^{-1}=\!\!\Rightarrow Q\circ(\zeta,I_*(\zeta),\theta)^{-1}$, $N\to\infty$.
Here, $(\zeta,W,\theta)$ (resp. $(\zeta_N,W_N,\theta_N)$) is the coordinate process corresponding to $Q$ (resp. $Q_N$).
\end{lemma}

\subsection{Convergence of Sampled Objective Functionals}\label{sec:limitsample}

In this section, we prove the convergence, as the number of samples $N\to\infty$, of the sampled objective functional $J_N(Q)$ given by \eqref{eq:control2Q}, for a fixed $Q\in{\cal Q}(\nu)$. Such an analysis uses the generalized convergence result given in Theorem~\ref{thm:limitP}.

For fixed $Q\in{\cal Q}(\nu)$, let $(\zeta,W,\theta)$ be the canonical (or coordinate) process corresponding  to $Q$. For $i\geq1$, recall that ${X}^{\theta,i}=(X^{\theta,i}(t))_{t\in[0,T]}$ solves the SDE~\eqref{eq:modeli} driven by $(\zeta,W,\theta)$. Next, we introduce a new empirical measure-valued process given by
\begin{align}\label{eq:hatmuemp}
\hat{\mu}^N(t) := \frac{1}{N}\sum_{i=1}^N\delta_{(\varepsilon^i,X^{\theta,i}(t))},\quad{\rm for}\ t\in[0,T].
\end{align}
The empirical process $\hat{\mu}^N=(\hat{\mu}^N(t))_{t\in[0,T]}$ can be viewed as the counterpart of $\mu^N$ defined in \eqref{eq:empiricalpmN}, but driven by $(\zeta,W,\theta)$ instead of the sample $(\zeta_N,W_N,\theta_N)$ from $Q_N\in{\cal Q}(\nu)$. We then define the law of $\hat{\mu}^N$ as:
\begin{align}\label{eq:lawhatmuemp}
\hat{\Qx}^N:=Q\circ(\hat{\mu}^N)^{-1}.
\end{align}
Using \eqref{eq:hatmuemp} and \eqref{eq:lawhatmuemp}, we may rewrite the sampled objective functional $J_N(Q)$ in \eqref{eq:control2Q} as follows:
\begin{align}\label{eq:control25}
J_N(Q)&=\ExQ\left[\lc\hat{\mu}^N(T),L\rc\right]+\frac{\beta}{\alpha}\ExQ\left[\int_0^T \lc\hat{\mu}^N(t),L\rc dt\right]+\ExQ\left[\int_0^T\{\lambda_1|\theta(t)|^2+\lambda_2|\theta'(t)|^2\}dt\right]\\
&=\int_{\hat{S}}\lc \vartheta(T),L\rc\hat{\Qx}^N(d\vartheta)+\frac{\beta}{\alpha}\int_0^T\left(\int_{\hat{S}}\lc\vartheta(t),L\rc\hat{\Qx}^N(d\vartheta)\right)dt
+\ExQ\left[\int_0^T\{\lambda_1|\theta(t)|^2+\lambda_2|\theta'(t)|^2\}dt\right].\notag
\end{align}
In the above expression, we recall that $\lc\mu,f\rc:=\int fd\mu$ for $\mu\in{\cal P}(E)$ and the loss function is defined by $L(e)=\alpha|x-y|^2$ where {$e=(\varepsilon,y,x)\in E$}. By applying Lemma~\ref{lem:Wtouni} and Corollary~\ref{coro:uniquemuN}, we immediately get  the following result.
\begin{lemma}\label{lem:Qmuthetastar}
Let {\Afg}, {\Atheta} and {\Anu} hold. Then, the precompact sequence $(\hat{\Qx}^N)_{N=1}^{\infty}\subset{\cal P}_2(\hat{S})$ has a unique limit point $\hat{\Qx}^*\in{\cal P}_2(\hat{S})$ satisfying ${\cal W}_{\hat{S},2}(\hat{\Qx}^N,\hat{\Qx}^*)\to0$ as $N\to\infty$.
\end{lemma}
{Moreover, the limit point $\hat{\Qx}^*\in{\cal P}_2(\hat{S})$ can be explicitly characterized, as shown in the following lemma whose proof is reported in the Appendix.}
\begin{lemma}\label{lem:limitQstar}
Let {\Afg}, {\Atheta} and {\Anu} hold. Assume $(Q_N)_{N=1}^{\infty},Q\subset{\cal Q}(\nu)$ satisfy $\lim_{N\to\infty}{\cal W}_{\Omega_{\infty},2}(Q_N,Q)=0$. Then $\Qx^N$ defined by \eqref{eq:QxN} converges to $Q\circ(I_*,\theta,\mu_*)^{-1}$ in ${\cal P}_2({\cal P}_2(E)\times{\cal C}_m\times\hat{S})$, as $N\to\infty$, where $I_*$ is given in {\Anu}, and $Q$-a.s., $\mu_*$ is the unique solution of FPK equation: for all $\varphi\in{\cal D}$,
\begin{align}\label{eq:mustar1}
\left\{
  \begin{array}{ll}
    \displaystyle \lc\mu_*(t),\varphi(t)\rc - \lc\mu_*(0),\varphi(0)\rc - \int_0^t\lc\mu_*(s),{{\cal A}^{\theta,\lc\mu_*(s),\rho\rc}}\varphi(s)\rc ds=0, & t\in(0,T];\\[0.6em]
\displaystyle \mu_*(0) =I_*.
  \end{array}
\right.
\end{align}
Moreover, it holds that $\hat{\Qx}^*=Q\circ(I_*,\theta,\mu_*)^{-1}$.
\end{lemma}

For a given $Q\in{\cal Q}(\nu)$ and the unique limit point $\hat{\Qx}^*\in{\cal P}_2(\hat{S})$ from Lemma~\ref{lem:limitQstar}, we define
\begin{align}\label{eq:control25JQ}
J(Q)&:=\int_{\hat{S}}\lc \vartheta(T),L\rc\hat{\Qx}^*(d\vartheta)+\frac{\beta}{\alpha}\int_0^T\left(\int_{\hat{S}}\lc\vartheta(t),L\rc\hat{\Qx}^*(d\vartheta)\right)dt+\ExQ\left[\int_0^T\{\lambda_1|\theta(t)|^2+\lambda_2|\theta'(t)|^2\}dt\right].
\end{align}
By Lemma \ref{lem:estimateXN} in the Appendix, we then have that $\sup_{N\geq1}J_N(Q)<+\infty$ for each $Q\in{\cal Q}(\nu)$. We are now ready to state the main result of this section:
\begin{proposition}\label{prop:JNlimitJ}
Let {\Afg}, {\Atheta} and {\Anu} hold. Then, for any $Q\in{\cal Q}(\nu)$,
\begin{align}\label{eq:JNQJQ00}
\lim_{N\to\infty}J_N(Q)=J(Q),
\end{align}
where $J_N(Q)$ and $J(Q)$ for $Q\in{\cal Q}(\nu)$ are defined by  \eqref{eq:control25} and \eqref{eq:control25JQ} respectively.
\end{proposition}

\begin{proof}
{To prove the proposition, we apply Lemma~\ref{lem:Qmuthetastar} and Theorem~7.12 of \cite{Villani2003}.}
Let $t\in[0,T]$, and define ${\cal L}^t(\vartheta):=\langle\vartheta(t),L\rangle=\int_{E}L(e)\vartheta(t,de)$ for all $\vartheta\in\hat{S}=C([0,T];{\cal P}_2(E))$. First of all, it follows from \eqref{eq:dhatS} that
\begin{align}
\left|{\cal L}^t(\vartheta)\right|=\int_{E}L(x,y)\vartheta(t,de)\leq 2\int_{E}|e|^2\vartheta(t,de)\leq 2\sup_{t\in[0,T]}{\cal W}_{E,2}^2(\vartheta(t),\delta_0)=2d_{\hat{S}}^2(\vartheta,\delta_0),
\end{align}
where {$e=(\varepsilon,y,x)\in E$}. {This shows that ${\cal L}^t$ satisfies the growth condition on $(\hat{S},d_{\hat{S}})$ with $p=2$ in Theorem~7.12-(iv) of \cite{Villani2003}}. Next, assume that $(\vartheta_l)_{l\geq1}\subset\hat{S}$ satisfy $\vartheta_l\to\vartheta$ on $(\hat{S},d_{\hat{S}})$, as $l\to\infty$. This implies that $\sup_{t\in[0,T]}{\cal W}_{E,2}(\vartheta_l(t),\vartheta(t))\to0$ as $l\to\infty$. Using Theorem~7.12 of \cite{Villani2003}, it follows that for any continuous function $\phi$ on $E$ satisfying the quadratic growth, $\langle\vartheta_l(t),\phi\rangle\to\langle\vartheta(t),\phi\rangle$ as $l\to\infty$. Note that $\phi(e):=|x-y|^2\leq2|e|^2$ and hence ${\cal L}^t(\vartheta_l)\to{\cal L}^t(\vartheta)$ as $l\to\infty$. Thus, we have shown that ${\cal L}^t$ is continuous and satisfies the quadratic growth on $(\hat{S},d_{\hat{S}})$. By Lemma~\ref{lem:Qmuthetastar},  ${\cal W}_{\hat{S},2}(\hat{\Qx}^N,\hat{\Qx}^*)\to0$ as $N\to\infty$. Again, by Theorem~7.12 of \cite{Villani2003}, we conclude that
\begin{align}\label{eq:gNtconver}
g_N(t):=\int_{\tilde{S}}{\cal L}^t(\vartheta)\hat{\Qx}^N(d\vartheta)\to \int_{\hat{S}}{\cal L}^t(\vartheta)\hat{\Qx}^*(d\vartheta),\quad N\to\infty.
\end{align}
Moreover, for $\epsilon>0$, using Jensen's inequality and Lemma \ref{lem:estimateXN}, we deduce the existence of a positive constant $C_{\epsilon,T}$ which only depends on $\epsilon,T$ such that
\begin{align}
&\sup_{N\geq1}\int_0^T \left|g_N(t)\right|^{1+\epsilon/2}dt=\sup_{N\geq1}\int_0^T \left|\ExQ\left[\langle\hat{\mu}^N(t),L\rangle\right]\right|^{1+\epsilon/2}dt\leq\sup_{N\geq1}\ExQ\left[\int_0^T\left|\langle\hat{\mu}^N(t),L\rangle\right|^{1+\epsilon/2}dt\right]\\
&\quad\leq\sup_{N\geq1}\frac{T}{N}\sum_{i=1}^N\ExQ\left[\sup_{t\in[0,T]}\left|X^{\theta,i}(t)-Y^i(0)\right|^{2+\epsilon}\right]\leq C_{\epsilon,T}\left\{\sup_{N\geq1}
\frac{1}{N}\sum_{i=1}^N\ExQ\left[\left\|X^{\theta,i}\right\|_T^{2+\epsilon}\right]+K^{2+\epsilon}\right\}<+\infty.\nonumber
\label{eq:intgNconv}
\end{align}
This implies that, as $R\to\infty$,
\begin{align}
\sup_{N\geq1}\int_{\{t\in[0,T];~|g_N(t)|\geq R\}}|g_N(t)|dt\leq \frac{1}{R^{\epsilon/2}}\sup_{N\geq1}\int_{\{t\in[0,T];~|g_N(t)|\geq R\}}|g_N(t)|^{1+\epsilon/2}dt\to 0.
\end{align}
Then, by Vitali's convergence theorem together with \eqref{eq:gNtconver} and \eqref{eq:intgNconv}, it follows that
\begin{align}\label{eq:intgNtconv}
\int_0^Tg_N(t)dt=\int_0^T\left(\int_{\tilde{S}}{\cal L}^t(\vartheta)\hat{\Qx}^N(d\vartheta)\right) dt\to \int_0^T\left(\int_{\hat{S}}{\cal L}^t(\vartheta)\hat{\Qx}^*(d\vartheta)\right) dt,\quad N\to\infty.
\end{align}
The desired convergence then follows from \eqref{eq:gNtconver} and \eqref{eq:intgNtconv}, recalling the expressions of $J_N(Q)$ and $J(Q)$ given, respectively, by \eqref{eq:control25} and \eqref{eq:control25JQ}.
\end{proof}

\section{Convergence of Minimizers of Sampled Objective Functionals}\label{sec:Gamma-convergence}

In this section, we show that the sequence of minimizers of the sampled objective functionals converges to the minimizer of the limiting objective functional, if $N$ is large enough. To establish this result, a {key step} is to prove the so-called $\Gamma$-convergence of $J_N$ to $J$ (see \eqref{eq:control25} and \eqref{eq:control25JQ}).

Before introducing the main result of this section, we first metrize the space ${\cal Q}(\nu)\subset{\cal P}_2(\Omega_{\infty})$ by taking the quadratic Wasserstein distance ${\cal W}_{\Omega_{\infty},2}$ on ${\cal Q}(\nu)$. The main result of the paper is as follows:
\begin{theorem}\label{thm:miniconver}
Let {\Afg}, {\Atheta} and {\Anu} hold. Then, it holds that
\begin{align}\label{eq:convJNJ}
\inf_{Q\in{\cal Q}(\nu)}J_N(Q)\to\inf_{Q\in{\cal Q}(\nu)}J(Q),\qquad N\to\infty,
\end{align}
where the minimum of $J(Q)$ over $Q\in{\cal Q}(\nu)$ exists. Moreover, if the minimizing sequence $(Q_N)_{N=1}^{\infty}\subset{\cal Q}(\nu)$ (up to a subsequence) converges to some $Q^*\in{\cal Q}(\nu)$ in ${\cal W}_{\Omega_{\infty},2}$, then $Q^*$ minimizes $J(Q)$ over $Q\in{\cal Q}(\nu)$.
\end{theorem}

\begin{proof}
The proof of Theorem~\ref{thm:miniconver} requires proving (i) the Gamma-convergence of $J_N$ to $J$, which is done in Proposition~\ref{thm:convergenceJNthetaN}; and (ii) that the minimizing sequence of the sampled optimization problem \eqref{eq:relaxed-controlQ} is precompact in ${\cal W}_{\Omega_{\infty},2}$, which is shown in Lemma~\ref{lem:precompactQN}.
\end{proof}

We next {give the definition of} Gamma-convergence of the sequence of sampled objective functionals $(J_N)_{N=1}^{\infty}$ on $({\cal Q}(\nu),\WINFTY)$ (see, e.g.~\cite{Dalmaso1993}):
\begin{definition}\label{def:Gammaconvergence}
$J_N:{\cal Q}(\nu)\to\R$ Gamma-converges to some functional $J:{\cal Q}(\nu)\to\R$, i.e., $J=\Gamma\mbox{-}\lim_{N\to\infty}J_N$ on ${\cal Q}(\nu)$, if the following conditions hold:
\begin{itemize}
  \item[{\rm\bf(i)}]~{\rm\bf(liminf~inequality):} For any $Q\in{\cal Q}(\nu)$ and every sequence $(Q_N)_{N=1}^{\infty}$ converging to $Q$ in $({\cal Q}(\nu),\WINFTY)$, we have that $\liminf_{N\to\infty}J_N(Q_N)\geq J(Q)$;
  \item[{\rm\bf(ii)}]~{\rm\bf(limsup~inequality):}  For any $Q\in{\cal Q}(\nu)$, there exists a sequence $(\bar{Q}_N)_{N=1}^{\infty}$ which converges to $Q$ in $({\cal Q}(\nu),\WINFTY)$ (this sequence is said to be a $\Gamma$-realising sequence), such that $\limsup_{N\to\infty}J_N(\bar{Q}_N)\leq J(Q)$.
\end{itemize}
\end{definition}

The following proposition shows that $J_N$ Gamma-converges to $J$ as $N\to\infty$.
\begin{proposition}\label{thm:convergenceJNthetaN}
Let {\Afg}, {\Atheta} and {\Anu} hold. Then $J=\Gamma\mbox{-}\lim_{N\to\infty}J_N$ on $({\cal Q}(\nu),\WINFTY)$.
\end{proposition}

\begin{proof}
Let $Q\in\mathcal{Q}(\nu)$ and take $\bar{Q}_N=Q$ for all $N\geq1$. Then, it follows from Proposition \ref{prop:JNlimitJ} that $\lim_{N\to\infty}J_N(\bar{Q}_N)=J(Q)$. Therefore, $(\bar{Q}_N)_{N=1}^{\infty}\subset{\cal Q}(\nu)$ is a $\Gamma$-realising sequence. Hence, the $\limsup$~inequality in Definition \ref{def:Gammaconvergence} holds.

It remains to prove the $\liminf$~inequality. Let $(Q_N)_{N=1}^{\infty},Q\subset{\cal Q}(\nu)$ satisfy $\lim_{N\to\infty}{\WINFTY}(Q_N,Q)=0$. Then, it follows from Lemma \ref{lem:limitQstar} that $\Qx^N=Q_N\circ(\mu^{(N)}(0),\theta_N,\mu^N)^{-1}$ converges to $Q\circ(I_*,\theta,\mu_*)^{-1}$ in ${\cal P}_2({\cal P}_2(E)\times{\cal C}_m\times\hat{S})$. The exact expression of $(I_*,\mu_*)$ is given in Lemma \ref{lem:limitQstar}. Recall the expression of ${\Qx}^N_{\mu}$ given in \eqref{eq:muQxN}. Then ${\Qx}^N_{\mu}$ converges to $\hat{\Qx}^*:=Q\circ\mu_*^{-1}$ in ${\cal P}_2(\hat{S})$, as $N\to\infty$. Using similar arguments to those in the proof of Proposition~\ref{prop:JNlimitJ}, this leads to
\begin{align}
&\lim_{N\to\infty}\int_{\hat{S}}\lc\vartheta(T),L\rc{\Qx}^N_{\mu}(d\vartheta)+\lim_{N\to\infty}\frac{\beta}{\alpha}\int_0^T\left(\int_{\hat{S}}\lc\vartheta(t),L\rc {\Qx}^N_{\mu}(d\vartheta)\right)dt\notag\\
&\qquad\qquad=\int_{\hat{S}}\lc \vartheta(T),L\rc\hat{\Qx}^*(d\vartheta)+\frac{\beta}{\alpha}\int_0^T\left(\int_{\hat{S}}\lc\vartheta(t),L\rc \hat{\Qx}^*(d\vartheta)\right)dt.
\label{eq:gamma1_}
\end{align}

We next state and prove the following claim:
\begin{align}\label{eq:fatou}
\ell:=\liminf_{N\to\infty}{\Ex}^{Q_N}\left[\left\|\theta_N\right\|_{{\cal L}_m^2}^2+\left\|\theta_N'\right\|_{{\cal L}_m^2}^2\right]\geq{\Ex}^{Q}\left[\left\|\theta\right\|_{{\cal L}_m^2}^2+\left\|\theta'\right\|_{{\cal L}_m^2}^2\right].
\end{align}
If $\ell=+\infty$, then \eqref{eq:fatou} trivially holds. If $\ell<+\infty$ (note that $\ell\geq0$ by the way it is defined), then passing to a subsequence (call it $Q_N$ again), we may assume that
\begin{align}\label{eq:assuminflim}
\lim_{N\to\infty}{\Ex}^{Q_N}\left[\left\|\theta_N\right\|_{{\cal L}_m^2}^2+\left\|\theta_N'\right\|_{{\cal L}_m^2}^2\right]=\ell<+\infty.
\end{align}
Note that $Q_N\circ\theta_N^{-1}\Rightarrow Q\circ\theta^{-1}$ as $N\to\infty$. Using Skorokhod's representation theorem, there exists a probability space $(\Omega^*,\F^*,\Px^*)$, a sequence of ${\cal C}_m$-valued r.v.s $(\theta_N^*)_{N=1}^{\infty}$, $\theta^*$ such that $\Px^*\circ(\theta_N^*)^{-1}= Q\circ\theta_N^{-1}$, $\Px^*\circ(\theta^*)^{-1}= Q\circ\theta^{-1}$, and as $N\to\infty$, $\theta_N^*\to\theta^*$ in ${\cal C}_m$, $\Px^*$-a.s. It follows from the dominated convergence theorem that
\begin{align}\label{eq:gamma2}
\lim_{N\to\infty}\Ex^{Q_N}\left[\left\|\theta_N\right\|_{{\cal L}_m^2}^2\right]&=\lim_{N\to\infty}\Ex^{*}\left[\left\|\theta_N^*\right\|_{{\cal L}_m^2}^2\right]=\Ex^{*}\left[\left\|\theta^*\right\|_{{\cal L}_m^2}^2\right]=\Ex^{Q}\left[\left\|\theta\right\|_{{\cal L}_m^2}^2\right].
\end{align}

We next prove that $({\theta_N^*}')_{N=1}^{\infty}$ is bounded in $L^2((0,T)\times\Omega^*;dt\otimes d\Px^*)$. This can be derived using similar arguments to those employed to derive~\eqref{eq:thetaderifatou}. 
For completeness, we provide the precise mathematical details.  For any $\phi\in{\cal D}$, define ${\cal T}_N(\phi):=(\theta_N^*,\phi')$. Let $(\phi_l)_{l=1}^{\infty}\subset{\cal D}$ be dense in ${\cal L}_m^2$. Then, for each $N\geq1$,
\begin{align}
\Ex^*\left[\sup_{l\geq1}\frac{\left|(\theta_N^*,\phi_l')\right|}{\left\|\phi_l\right\|_{{\cal L}_m^2}}\right]
=\Ex^{Q_N}\left[\sup_{l\geq1}\frac{\left|(\theta_N',\phi_l)\right|}{\left\|\phi_l\right\|_{{\cal L}_m^2}}\right]=\Ex^{Q_N}\left[\left\|\theta_N'\right\|_{{\cal L}_m^2}\right]<+\infty.
\end{align}
By Hahn-Banach theorem and Riesz representation theorem, there exists an ${\cal L}_m^2$-valued r.v. $\hat{\theta}_N^*$ such that ${\cal T}_N(\phi)=(\hat{\theta}_N^*,\phi)$ for all $\phi\in {\cal L}_m^2$, $\Px^*$-a.s. In particular, ${\cal T}_N(\phi)=(\theta_N^*,\phi')=(\hat{\theta}_N^*,\phi)$ for any $\phi\in{\cal D}$. This yields that $\theta_N^*\in{\cal H}_m^1$, $\Px^*$-a.s. It then follows from \eqref{eq:assuminflim} that
\begin{align}\label{eq:thetastarL2norm}
\sup_{N\geq1}\Ex^*\left[\left\|{\theta_N^*}'\right\|_{{\cal L}_m^2}\right]&=\sup_{N\geq1}\Ex^*\left[\sup_{l\geq1}\frac{|(\theta_N^*,\phi_l')|}{\|\phi_l\|_{{\cal L}_m^2}}\right]
=\sup_{N\geq1}\Ex^{Q_N}\left[\sup_{l\geq1}\frac{|(\theta_N,\phi_l')|}{\|\phi_l\|_{{\cal L}_m^2}}\right]=\sup_{N\geq1}\Ex^{Q_N}\left[\left\|\theta_N'\right\|_{{\cal L}_m^2}\right]<\infty.
\end{align}
This shows that $({\theta_N^*}')_{N=1}^{\infty}$ is bounded in $L^2((0,T)\times\Omega^*;dt\otimes d\Px^*)$, and hence ${\theta_N^*}'$ (up to a subsequence) converges weakly to some $\hat{\theta}^*\in L^2((0,T)\times\Omega^*;dt\otimes d\Px^*)$ as $N\to\infty$. As shown in Proposition~\ref{prop:reaxed-sol}, for any $\phi\in{\cal D}$ and $H\in L^\infty(\Omega^*;\Px^*)$, by the weak convergence property
\begin{align}
\Ex^*\left[(\theta^*,\phi')H\right]=\lim_{N\rightarrow\infty}\Ex^*\left[ (\theta^*_N,\phi')H\right]=-\lim_{N\rightarrow\infty}\Ex^*\left[({\theta^*_N}',\phi)H\right]=-\Ex^*\big[(\hat{\theta}^*,\phi)H\big].
\end{align}
Thus, for any $\phi\in{\cal D}$, $(\theta^*,\phi')=-(\hat{\theta}^*,\phi)$, $\Px^*$-a.s.. The separability of ${\cal D}$ implies that $\Px^*$-a.s., $(\theta^*,\phi')=-(\hat{\theta}^*,\phi)$, for all $\phi\in{\cal D}$, i.e., ${\theta^*}'=\hat{\theta}^*$, $\Px^*$-a.s.. Similarly to the derivation of the estimates~\eqref{eq:thetaderifatou} and \eqref{eq:thetastarL2norm}, we obtain that
\begin{align}\label{eq:gamma3}
\liminf_{N\to\infty}\Ex^{Q_N}\left[\left\|\theta'_N\right\|_{{\cal L}_m^2}^2\right]=\liminf_{N\to\infty}\Ex^*\left[\left\|{\theta^*_N}'\right\|_{{\cal L}_m^2}^2\right]\geq\Ex^*\left[\left\|{\theta^*}'\right\|_{{\cal L}_m^2}^2\right]=\Ex^Q\left[\left\|\theta'\right\|_{{\cal L}_m^2}^2\right].
\end{align}
Thus, the proof of~\eqref{eq:fatou} follows immediately from \eqref{eq:gamma2} and \eqref{eq:gamma3}. Finally, note that
\begin{align}
J_N(Q_N)=\int_{\hat{S}}\lc\vartheta(T),L\rc{\Qx}^N_{\mu}(d\vartheta)+\frac{\beta}{\alpha}\int_0^T\left(\int_{\hat{S}}\lc\vartheta(t),L\rc {\Qx}^N_{\mu}(d\vartheta)\right)dt+{\Ex}^{Q_N}\left[\left\|\theta_N\right\|_{{\cal L}_m^2}^2+\left\|\theta_N'\right\|_{{\cal L}_m^2}^2\right].
\end{align}
Then, the $\liminf$~inequality in Definition \ref{def:Gammaconvergence} follows from \eqref{eq:control25JQ}, \eqref{eq:gamma1_}, and \eqref{eq:fatou}.
\end{proof}

To conclude the proof of Theorem~\ref{thm:miniconver}, it remains to establish the relatively compactness of the minimizing sequence of the sampled optimization problem~\eqref{eq:relaxed-controlQ}. As shown in Proposition~\ref{prop:reaxed-sol} of Section~\ref{sec:existence-sol-samploptim}, for each $N\geq1$  there exists a relaxed solution $Q_N\in{\cal Q}(\nu)$ such that $J_N(Q_N)=\inf_{Q\in{\cal Q}(\nu)}J_N(Q)$. We will prove below that such a sequence $(Q_N)_{N=1}^{\infty}$ is precompact under the quadratic Wasserstein distance $\WINFTY$.
\begin{lemma}\label{lem:precompactQN}
Let {\Afg} and {\Atheta} hold. Then, the above minimizing sequence $(Q_N)_{N=1}^{\infty}\subset{\cal Q}(\nu)$ is precompact in $\WINFTY$.
\end{lemma}
\begin{proof}
	Lemma~\ref{lem:estimateXN} implies that $\sup_{N\geq1}J_N(Q)\leq \ell$ for some $\ell>0$. Together with Proposition~\ref{prop:reaxed-sol}, this implies the existence of a constant $\ell>0$ and some $Q\in{\cal Q}(\nu)$ such that
	\begin{align}
	J_N(Q_N)=\inf_{Q\in{\cal Q}(\nu)}J_N(Q)\leq \sup_{N\geq1}J_N(Q)\leq \ell,\quad\forall\ N\geq1.
	\end{align}
	Therefore, it follows that
	\begin{align}
	\sup_{N\geq1}\Ex^{Q_N}\left[\left\|\theta_N\right\|^2_{{\cal L}_m^2}+\left\|\theta_N'\right\|^2_{{\cal L}_m^2}\right]\leq \sup_{N\geq1}J_N(Q_N)\leq \ell.
	\end{align}
	Let ${\cal X}_N=(\zeta_N,W_N,\theta_N)$ be the canonical process corresponding to $Q_N$. Using similar arguments to those in the proof of Proposition~\ref{prop:reaxed-sol}, $(Q_N\circ\theta^{-1}_N)_{N=1}^{\infty}$ is tight on ${\cal C}_m$, and moreover $(Q_N\circ{\cal X}^{-1}_N)_{N=1}^{\infty}$ is tight on $\Omega_\infty$ because $(\Omega_{\infty},d)$ is Polish. Then, by Prokhov's theorem, there exists a $Q^*\in{\cal P}(\Omega_{\infty})$ such that the minimizing sequence $(Q_N)_{N=1}^{\infty}$, up to a subsequence, converges to $Q^*$ under weak topology. Using Skorokhod representation theorem, there exists a probability space $(\Omega^*,\F^*,\Px^*)$, ${\cal X}_N^*=(\zeta^*_N,W_N^*,\theta_N^*)$ with ${\cal X}_N^*\overset{d}{=}{\cal X}_N$, and ${\cal X}^*=(\zeta^*,W^*,\theta^*)$ with $\Px^*\circ({\cal X}^*)^{-1}=Q^*$ such that, $\Px^*$-a.s., as $N\to\infty$,
	\begin{align}\label{eq:Pstarconver3}
	\zeta_N^*&\to \zeta^*\ {\rm in}\ \Xi_K^{\N};\quad W_N^*\to W^*\ {\rm in}\ {\cal C}_p^{\N};\quad
	\theta_N^*\to \theta^*\ {\rm in}\ {\cal C}_m.
	\end{align}
	Further, we obtain that {\bf(i)}: for $N\geq1$, $\theta_N^*\in{\cal H}_m^1$ $\Px^*$-a.s. and hence $\Px^*\circ({\cal X}_N^{*})^{-1}\in{\cal P}(\Omega_{\infty})$; {\bf(ii)}: $Q^*=\Px^*\circ({\cal X}^{*})^{-1}\in{\cal P}(\Omega_{\infty})$, and hence $Q^*\in{\cal Q}(\nu)$. To prove that $Q_N$ converges to $Q^*$ as $N\to\infty$ in $\WINFTY$, using Definition 6.8 and Theorem 6.9 in \cite{Villani2009}, it suffices to prove that
	\begin{align}\label{eq:P1Omegainfty}
	&\lim_{N\to\infty}\int_{\Omega_{\infty}}d_{\infty}^2((\gamma,w,\vartheta),(\hat{\gamma},\hat{w},\hat{\vartheta}))Q_N(d(\gamma,w,\vartheta))=\int_{\Omega_{\infty}}
	d_{\infty}^2((\gamma,w,\vartheta),(\hat{\gamma},\hat{w},\hat{\vartheta}))Q^*(d(\gamma,w,\vartheta)),
	\end{align}
	for some $(\hat{\gamma},\hat{w},\hat{\vartheta})\in\Omega_{\infty}$. First, note that $Q_N=\Px^*\circ({\cal X}_N^*)^{-1}$. Then, for all $N\geq1$, it holds that
	\begin{align}
	\int_{\Omega_{\infty}}d_{\infty}^2((\gamma,w,\vartheta),(\hat{\gamma},\hat{w},\hat{\vartheta}))Q_N(d(\gamma,w,\vartheta))
	=\Ex^*\left[\left|d_1(\zeta_N^*,\hat{\gamma})+d_2(W_N^*,\hat{w})+d_3(\theta_N^*,\hat{\vartheta})\right|^2\right].
	\end{align}
	Observing that, for any $N\geq1$, $(\zeta^*_N,W^*_N)$ and $(\zeta^*,W^*)$ are identically distributed by (i) and (ii) of Definition \ref{def:relax-sol}, it follows that $(d^2_1(\zeta_N^*,\hat{\gamma})+d^2_2(W_N^*,\hat{w}))_{N=1}^{\infty}$ is uniformly integrable. Using {\Atheta}, we deduce that $(d^2_3(\theta_N^*,\hat{\vartheta}))_{N=1}^{\infty}$, is uniformly integrable. It then follows from \eqref{eq:Pstarconver3} and Vitali's convergence theorem that
	\begin{align}
	&\lim_{N\to\infty}\int_{\Omega_{\infty}}d_{\infty}^2((\gamma,w,\vartheta),(\hat{\gamma},\hat{w},\hat{\vartheta}))Q_N(d(\gamma,w,\vartheta))
	=\lim_{N\to\infty}\Ex^*\left[\left|d_1(\zeta_N^*,\hat{\gamma})+d_2(W_N^*,\hat{w})+d_3(\theta_N^*,\hat{\vartheta})\right|^2\right]\nonumber\\
	&\qquad=\Ex^*\left[\left|d_1(\zeta^*,\hat{\gamma})+d_2(W^*,\hat{w})+d_3(\theta^*,\hat{\vartheta})\right|^2\right]
	=\int_{\Omega_{\infty}}d_{\infty}^2((\gamma,w,\vartheta),(\hat{\gamma},\hat{w},\hat{\vartheta}))Q^*(d(\gamma,w,\vartheta)),
	\end{align}
	which yields \eqref{eq:P1Omegainfty}. This completes the proof of the lemma.
\end{proof}

\begin{remark}\label{eq:lambda2eqn0}
If $\lambda_2=0$ in \eqref{eq:squareLR}, one can apply a similar argument to the one used in the case $\lambda_2>0$ to characterize the limiting objective $J(Q)=\lim_{N\to\infty} J_N(Q)$ and the Gamma convergence $J_N\overset{\Gamma}{\to}J$ as $N\to\infty$. However, without derivative-based regularization, one needs a different argument to esablish the convergence of minimizers as the sample size grows to infinity. See Appendix~\ref{app:proof1} for more details, {where we require the forcing function $f$ specified in Section~\ref{sec:concludes} to satisfy the condition~\eqref{eq:dynamf}  }.		
\end{remark}

\section{Connection to Deep ResNet} \label{sec:appresnet}
In this section, we demonstrate how the proposed framework can mimic the training process of {ResNet}s. Section~\ref{discretecont}  illustrates how a stylized unregularized ResNet with infinite depth can be modeled by a continuous time deterministic dynamical system. Section~\ref{sec:regularizedeep} shows how, extending the deterministic to a stochastic system of the form in Eq.~\eqref{eq:modeli}, one can model regularization during the training process. Section~\ref{sec:losses} discusses how the optimization criterion considered in the previous sections is related to regularized loss functions of neural networks.

\subsection{Representation of Deep {ResNet} as a Dynamical System}\label{discretecont}

We begin by considering a simplified version of the $n$-layered ResNet architecture, as in \cite{ThGe2018}. For $l=0,1,\ldots,n-1$, let $d$ be the number of neurons in each layer, and $X^{(n)}(l)\in\R^d$ the states of neurons in layer $l$. We use $w^{(n)}(l)\in\R^{d\times d}$ to denote the matrix of weights which determine how neurons in layer $l$ activate neurons in layer $l+1$, and $b^{(n)}(l) \in \R^d$ to denote the  vector of biases at layer $l$. The feed-forward propagation in the $n$-layer ResNet model can be represented by the difference equation:
\begin{align}\label{eq:model0_}
	X^{(n)}(l+1)= X^{(n)}(l) + \frac{1}{n}\sigma\left(l,w^{(n)}(l)X^{(n)}(l)+b^{(n)}(l)\right),\quad l=0,1,\ldots,n-1,
\end{align}
where $\frac{1}{n}$ is the scaling factor, and $\sigma(l,x)=(\sigma_1(l,x_1),\ldots,\sigma_d(l,x_d))^{\top}$ is the activation function which effectively turns neurons ``on'' or ``off'' at layer $l+1$ based on the value of the input $x = (x_1,\ldots,x_d)$ at layer $l$. In existing implementations, the activation function is chosen to be a smooth approximation of a rectified linear activation function or a sigmoid function (see also \cite{Haykin2009}). Eq.~\eqref{eq:model0_} highlights the residual property of the network: the endogenous input at layer $l+1$ consists of the non-transformed input $X^{(n)}(l)$ from layer $l$, plus a nonlinear transformation of $X^{(n)}(l)$. The term $X^{(n)}(l)$ represents information from the previous layer ``skipping the processing associated with the layer $l$'', and being transmitted to layer $l+1$ without being transformed. Let $\theta^{(n)}(l):=(w^{(n)}(l),b^{(n)}(l))\in\R^{d\times d}\times\R^d$, and write
\begin{align}
	g(l,\theta^{(n)}(l),x):=\sigma\left(l,w^{(n)}(l)x+b^{(n)}(l)\right).
\end{align}
As observed in prior studies (e.g. \cite{E17} and \cite{Lu}), one can turn the setting described above into an explicit Euler characterization of the ODE, yielding
\begin{align}\label{eq:model11}
	dX^{\theta}(t)= g(t, \theta(t), X^{\theta}(t))dt,\qquad t\in(0,T],
\end{align}
where $X^{\theta}(t) \in \R^d$ is the state at layer $t$ of the feed-forward transformation triggered by the the training inputs, $t\in[0,T]$ represents the depth of the network, and the time step $dt$ corresponds to the scaling factor $\frac{1}{n}$. The initial state $X^{\theta}(0)=X(0)$ collects the training inputs.~
Above, we have used the superscript $\theta$ to emphasize the dependence of output neurons on the weight matrix and, for a finite $T$, $X^{\theta}$ and $\theta$ denote real valued functions on $[0,T]$. The above argument has been used to justify neural network architectures arising from discretization of differential equations (e.g. \cite{chenetal2018} and \cite{ThGe2018}).
We will treat~\eqref{eq:model11} as a differential equation, i.e., assume an infinitesimal time step. This effectively corresponds to the limit of an infinitely deep residual network, i.e., consisting of infinitely many layers.

\subsection{Regularized Deep {ResNet}}\label{sec:regularizedeep}

We extend the dynamical system representation of a deep {ResNet} sketched in the previous section to incorporate regularization during the training process. Specifically, we allow for (i) additive noise injection into the deterministic hidden units to prevent overfitting, and (ii) batch normalization in each layer to accelerate convergence of the training algorithm. This turns the continuous time dynamics given by~\eqref{eq:model11} into a stochastic dynamical system of the form in Eq.~\eqref{eq:modeli}.

We next illustrate how the proposed framework can be used to describe continuous layer idealizations of deep ResNets at different levels of complexity.

\begin{itemize}
	\item The vector function $f$ in Eq.~\eqref{eq:modeli}
	 depends only on the input variables $(t,\theta,x)$, and  $\varepsilon^i=0$. This yields a feed forward propagation in a deep network, where the randomness only comes from the initial input $X^i(0)$. Such a dynamical system specification recovers the framework of \cite{EHanLi18}, which has also been specified in Eq.~\eqref{eq:model11}.

	\item {The vector function} $f$ depends on all its inputs {$(t,\theta,x,\eta)$}, and $\varepsilon^i=0$. We next illustrate how this specification incorporates batch normalization (see also Algorithm 1 in \cite{IoffeSzegedy2015}) through the batch function $\rho:\R^d\to\R^{2}$. For the purpose of illustration, we set $d=1$. The batch normalization transform applied to $X^{\theta,i}(t)$ over a batch of size $N$ is given by
	\begin{align}
		\gamma \frac{X^{\theta,i}(t)-\frac{1}{N}\sum_{j=1}^NX^{\theta,j}(t)}{\sqrt{\frac{1}{N}\sum_{i=1}^N(X^{\theta,i}(t))^2-(\sum_{j=1}^NX^{\theta,j}(t))^2+\iota}}+\upsilon,
	\end{align}
	where $\gamma,\upsilon\in\R$ are, respectively, the scale and shift factors and $\iota>0$ is a constant added to the batch variance to avoid numerical instability. Take the (vector) batch function $\rho(x)=(\rho_1(x),\rho_2(x))=(x,x^2)$. Then, we may rewrite the above batch normalization as:
	\begin{align}
		\frac{\gamma\left(X^{\theta,i}(t)-\frac{1}{N}\sum_{j=1}^N\rho_1(X^{\theta,j}(t))\right)}{\sqrt{\frac{1}{N}\sum_{j=1}^N\rho_2(X^{\theta,j}(t))
				-\left(\frac{1}{N}\sum_{j=1}^N\rho_1(X^{\theta,j}(t))\right)^2+\iota}}+\upsilon.
	\end{align}
	Because of such normalization, the state dynamics of the neural network becomes of the mean-field type. For $(t,\theta,x,\eta)=(t,(w,b),x,(\eta_1,\eta_2))\in[0,T]\times\R^2\times\R\times\R^2$, this yields
	\begin{align}\label{example0}
		f\left(t, (w,b), x,(\eta_1,\eta_2)\right) = S\left(t,w\frac{\gamma(x-\eta_1)}{\sqrt{\eta_2-\eta_1^2+\iota}}+w\upsilon+b\right),
	\end{align}
	where $x\to S(t,x)=\frac{1}{1+e^{-x}}$ is the logistic activation function at layer $t$. Observe that, differently from \cite{EHanLi18}, the probability law of the state process $(X^{\theta,1},\ldots,X^{\theta,N})$ enters explicitly into the state forward dynamics.

\begin{remark}\label{rem:thosimga}
	The batch function  $\rho_2(x)=|x|^2$ introduced in Algorithm 1 of \cite{IoffeSzegedy2015} is only locally Lipschitz. For practical matters, one can well approximate $\rho_2$ with the truncation function $\rho_2^{(M)}(x):=x^2{\bf1}_{\{|x|\leq M\}}+M^2{\bf1}_{\{|x|>M\}}$ where $M>0$ is sufficiently large. We can then consider an approximating system to \eqref{eq:modeli}, in which $\rho_2$ is replaced by $\rho_2^{(M)}$. Such a system is well-posed. Denote by $X^{\theta,i,M}=(X^{\theta,i,M}(t))_{t\in[0,T]}$ the solution of this approximating system. First, in terms of \eqref{example0}, for any $\theta\in\Theta$ (a compact subset of $\R^m$) and $x^i=(x_1^i,\ldots,x_d^i)\in\R^d$ ($i=1,\ldots,N$), we have that $a_i(t,(x^1,\ldots,x^N)):=f(t,\theta, x^i,(\frac1N\sum_{j=1}^N\rho_1(x^j),\frac1N\sum_{j=1}^N\rho_2(x^j)))$ satisfies the linear growth condition uniformly in $t\in[0,T]$. Since the locally Lipschitz continuity implies strong uniqueness, we have $X^{\theta,i}(t)=X^{\theta,i,M}(t)$ for $t\in[0,\tau_{M}^i]$ a.s., where $\tau^i_{M}:=\inf\{t\geq0;~|X^{\theta,i}(t)|\geq M\}$. {Because $a_i(t,\cdot)$ has  linear growth for any $t\in[0,T]$, it follows from the Gronwall's inequality that $\Ex[\sup_{t\in[0,T]}|X^{\theta,i}(t)|]<+\infty$. Then, for any $t\in[0,T]$, $\lim_{M\to+\infty}\Px(\tau^i_M\leq t)=\lim_{M\to+\infty}\Px(\sup_{s\in[0,t]}|X^{\theta,i}(s)|\geq M)\leq\lim_{M\to+\infty}\frac1M\Ex [\sup_{s\in[0,T]}|X^{\theta,i}(s)|]=0$. This implies that, for any $t\in[0,T]$, $\Px(\bigcap_{N=1}^{+\infty}\bigcup_{M=N}^{+\infty}\{\tau^i_M\leq t\})=0$. Since $\tau_M^i$ is increasing in $M>0$, we deduce that $\lim_{M\to+\infty}\tau^i_{M}=+\infty$, a.s..} Thus $|X^{\theta,i,M}(t)-X^{\theta,i}(t)|\to0$, as $M\to\infty$, a.s.. Hence, $\lim_{M\to+\infty}X^{\theta,i,M}(t)$ for $t\in[0,T]$ is the unique solution of \eqref{eq:modeli}.
\end{remark}	
	
	\item {The vector function} $f$ {depends on all its inputs $(t,\theta,x,\eta)$,} and $\varepsilon^i \neq 0$. This is the most general setting, in which a noise process modeled through the Brownian motion {$W(t)=(W^1(t),\ldots,W^N(t))_{t\in[0,T]}$} is added along the trajectory of the deep neural network. The weight function  $\theta=(\theta(t))_{t\in[0,T]}$ then becomes a $\Fx$-adapted control strategy, which will be chosen to minimize the neural network loss criterion. One can thus think of an arbitrary layer as receiving samples from a distribution that is determined by the layer below. Such a distribution changes during the course of training, making any layer except the first one responsible not only for learning a good representation, but also for adapting to a changing input distribution.

\qquad We sample noise from a zero-mean Gaussian distribution so that the noise transition function is unbiased. This means that, conditional on the input $X^{\theta,i}(t)$, the output $X^{\theta,i}(t+dt)$ at layer $t+dt$ coincides on average with the output of a deterministic ResNet. In other words, the noise-injected transition function preserves, on average, the transition function of the underlying deterministic network. This type of unbiased regularization has been shown to perform favorably compared to classical regularization techniques such as dropout, in the context of recurrent neural networks (see \cite{Blei}).	One can also regard the noise injected outputs of hidden units as stochastic activations, or equivalently, random selections of hidden units in a layer, consistently with \cite{NIPSStochAct}. A related study by \cite{RakinHeFan2018} considers noise injection as a way to make deep ResNets robust against adversarial attacks. They propose Parametric-Noise-Injection (PNI), i.e., they inject noise to different components or locations within the network.

\begin{remark}
It is possible to account for a multiplicative noise injection regularization procedure. Such regularization technique, commonly referred to as the dropout, has been successfully used to mitigate the overfitting problem of overparameterized networks (see, e.g. \cite{SrivastavaMLR} and \cite{Ba2013}). In our framework, this means replacing the term $\varepsilon^idW^i(t)$ in \eqref{eq:modeli} with $\varepsilon^i\sigma(t,\theta,X^{\theta,i}(t))dW^i(t)$, where $\sigma(t,\theta,x):[0,T]\times\Theta\times\R^d\to\R_+$ is the common multiplicative noise scale function. The main results developed in previous sections still hold if the following standard assumption on $\sigma$ is imposed:
\begin{itemize}
		\item[\Asi]
		\begin{enumerate}
			\item[{\rm(i)}] $\sigma(t,\theta,x)$ is continuous in $(t,\theta,x)$ and is twice continuously differentiable in $x$;
			\item[{\rm(ii)}] $\sigma$ satisfies the following Lipschitz condition: for all $\theta_1,\theta_2\in\Theta$ and $x,y\in\R^d$,
			\begin{align*}
				|\sigma(t,\theta_1,x)-\sigma(t,\theta_2,y)|\leq [\sigma]_{\rm Lip}[|\theta_1-\theta_2|+|x-y|],
			\end{align*}
		\end{enumerate}
		where $[\sigma]_{{\rm Lip}}$ is the Lipschitz coefficient of $\sigma$, which is independent of the layer index $t$.
\end{itemize}
We provide the corresponding proofs and supporting arguments in the Online Appendix of this paper.
\end{remark}

\end{itemize}

\subsection{Loss Function and Regularization} \label{sec:losses}
We can view the samples $(\zeta^i)_{i=1}^N$ in Eq.~\eqref{eq:zetai} as $N$ i.i.d. {\it training samples} of a neural network. The $i$-th {\it training sample} consists of the training input {$X^i(0)$}, and the label part $Y^i(0)$.

In a neural network, the loss function quantifies the difference between the expected outcome $Y^i(0)$ and the endogenous output $X^{\theta,i}(T)$ produced by the learning model, for all training samples $\zeta^i$, $i=1,\ldots, N$. A simple and commonly used loss function is the squared-error loss specified by Eq.~\eqref{eq:squareLR}, which makes it computationally tractable to derive the gradients used to update the weights. It is also common to consider a regularizer in the objective function to penalize for model complexity. The specification in Eq.~\eqref{eq:squareLR} is flexible enough to allow for two types of regularization. The $L^2$-regularization is achieved by setting $\lambda_2 = 0$, and encourages the weight values {\it towards} zero. The $L^1$-regularization is instead attained with $\lambda_1=0$, and encourages the weight values {\it to be} zero. Intuitively, smaller weights reduce the impact of hidden neurons, so that they become neglectable and the overall complexity of the neural network gets reduced.

{Regularizers in the objective function of the neural network have been extensively studied.  \cite{HasanRoy-Chowdhury15} consider the case where $\lambda_1>0$ and $\lambda_2=\beta=0$. If $\lambda_1=\lambda_2>0$, the control problem \eqref{eq:control2} includes an ${\cal H}_m^1$-regularizer, as in the studies of \cite{HaberRuthotto2018}, \cite{ThGe2018}, and \cite{Oberman2018} for the training of deep neural networks.} For instance, \cite{HaberRuthotto2018} consider a smoothness regularization on the weights $w$ given by $R(w)=\frac{1}{2}\|{\mathcal L}{\bf w}\|^2$, where ${\mathcal L}$ is a discretized differential operator. In our continuous time framework, this regularization can be understood as a discrete version of the time derivative of weights. They show that this regularization biases the learning process towards smooth time dynamics, and improves  generalization of a ResNet. \cite{Oberman2018} consider a modified loss function with Lipschitz regularization.  They show theoretically that the regularized model generalizes, and empirically that the regularization improves adversarial robustness for deep neural networks (DNN).

\section{Concluding Remarks and Future Directions}\label{sec:concludes}

We have studied a class of optimal sampled (relaxed) control problems of the mean-field type, and then demonstrated how they can be used to model the optimal training of a stlyzed continuous layer deep ResNet. Building upon the existence of optimal relaxed controls for  training sets of finite size, we have proven that the minimizer of the sampled (relaxed) problem converges to that of the limiting optimization problem, as the number of samples grows large.

Our work can be extended along several directions. One avenue of investigation is to study how the input dimension influences the convergence rate and the time discretization error. For mean field models without control, \cite{LuconStannat14} study the speed at which the empirical measure of a class of disordered diffusions with singular interactions converges to the solution of a deterministic McKean-Vlasov equation. \cite{BencheikhJourdain21} establish the rate of convergence for a system of interacting particles with mean-field rank-based interaction in the drift coefficient and constant diffusion coefficient. The core notion is the propagation of chaos for large systems of interacting particles. In the discrete time case, \cite{MottePhamAAP2021} show the role of relaxed controls in the study of Markov decision processes (MDP) under mean field interaction with common noise, and establish a rate of convergence for the problem of $N$-cooperative agents as $N$ grows to infinity. Their results can be leveraged to explore the speed of convergence of $J_N$ to $J$ and ${\rm argmin}~J_N$ to ${\rm argmin}~J$, as $N\to\infty$.

In the absence of derivative-based regularization (i.e., $\lambda_2=0$ in \eqref{eq:squareLR}), our results hold for the following linear type of forcing function:
\begin{align}\label{eq:dynamf}
  f(t,\theta,x,\eta)=f(t,(\theta^{(1)},\theta^{(2)}),x,\eta)=\theta^{(1)}S\left(t,x,\eta\right)+\theta^{(2)},
\end{align}
where $\theta^{(1)}\in\mathbb R^{d\times l}$, $\theta^{(2)}\in\mathbb R^d$. The function  $S(t,x,\eta):[0,T]\times\R^d\times\R^q\to\R^l$ is continuous in $(t,x,\eta)$, and it satisfies the inequality $\|S\|_{\infty}+\|\nabla_{x}S\|_{\infty}+\|\nabla_{\eta}S\|_{\infty}<+\infty$, where $\|\cdot\|_{\infty}$ denotes the super-norm.

Our limiting results have implications for the analysis of deep ResNets. They imply that the sampled control problems arising in the optimal training of a deep ResNet can be cast as a one-dimensional stochastic control problem. Despite the solvability of the resulting one-dimensional stochastic control problem is nontrivial, this reduction allows to bypass the complexity of a large sample size during the supervised learning process. In a future continuation of this work, it would be interesting to establish a connection between the limiting stochastic control problem $\inf_{Q\in{\cal Q}(\nu)}J(Q)$ (c.f.~Theorem~\ref{thm:miniconver}) and a class of deterministic FPK control problems described as follows: for some $\mu_0\in{\cal P}_2(E)$,
\begin{align}\label{eq:deterministicprob}
\underline{J}^d:=\min_{\theta\in U_m^d} J^d(\mu^{\theta},\theta),\quad { U_m^d:=\{\theta\in{\cal H}_m^1;\ \theta\in\Theta,~{\rm a.s.~on}~(0,T)\}},
\end{align}
subject to the constraint:
\begin{align}\label{eq:constraint}
\left\{
  \begin{array}{ll}
    \displaystyle \lc\mu^{\theta}(t),\varphi(t)\rc - \lc\mu^{\theta}(0),\varphi(0)\rc - \int_0^t\lc\mu^{\theta}(s),{{\cal A}^{\theta,\lc\mu^{\theta}(s),\rho\rc}}\varphi(s)\rc ds=0, & t\in(0,T];\\[0.6em]
\displaystyle \mu^{\theta}(0) =\mu_0,\ \forall\ \varphi\in{\cal D}.
  \end{array}
\right.
\end{align}
For $(\mu,\theta)\in\hat{S}\times{\cal H}_m^1$, the objective functional $J^d$ in~\eqref{eq:deterministicprob} is given by
\begin{align}\label{eq:objectiveJd}
J^d(\mu,\theta)&:=\lc\mu(T),L\rc+\frac{\beta}{\alpha}\int_0^T\lc\mu(t),L\rc dt+\int_0^T\{\lambda_1|\theta(t)|^2+\lambda_2|\theta'(t)|^2\}dt.
\end{align}
The deterministic optimal control in \eqref{eq:deterministicprob} is obtained by minimizing the objective function under the FPK equation constraint. This type of control problems has also been studied numerically using receding-horizon control techniques, see \cite{AnnunziatoBorzi2010}.

\section*{Acknowledgements}

We thank the two anonymous reviewers for their careful reading of our manuscript and their insightful comments and suggestions, which certainly helped to improve our manuscript. L. Bo was supported by the Natural Science Foundation of China (Grant 11971368). A. Capponi was supported in part by the Natural Science Foundation (Grant DMS-1716145).

\appendix
\section{Appendix}\label{app:proof1}
\renewcommand\theequation{A.\arabic{equation}}
\setcounter{equation}{0}

This Appendix provides proofs to some of the propositions and lemmas stated in the main body. Additionally, it provides a miscellaneous of technical results, along with the corresponding proofs, which are used to derive proofs of propositions and theorems stated in the main body of the paper.
\begin{proof}[Proof of Proposition~\ref{prop:reaxed-sol}]
Without loss of generality, we assume that $\alpha_N=\inf_{Q\in{\cal Q}(\nu)}J_N(Q)<+\infty$. In view of \eqref{eq:control2Q}, let $(Q_k)_{k=1}^{\infty}\subset{\cal Q}(\nu)$ be a minimizing sequence such that
\begin{align}\label{eq:minimizingse}
 0\leq \alpha_N\leq J_N(Q_k)\leq \alpha_N +\frac{1}{k},\quad \forall\ k\geq1.
\end{align}
Above, for $k\geq1$, the objective functional
\begin{align}\label{eq:JNQk00}
J_N(Q_k)=\Ex^{Q_k}\left[L_N(\tilde{X}_k^{\theta}(T),\tilde{Y}_k(0))+\int_0^TR_N(\theta_k(t),\theta_k'(t);\tilde{X}_k^{\theta}(t),\tilde{Y}_k(0))dt\right],
\end{align}
where ${\cal X}_k:=(\zeta_k,W_k,\theta_k)=((\zeta_k^i)_{i=1}^{\infty},(W_k^i)_{i=1}^{\infty},\theta_k)$ is the canonical process. Since $Q_k\in{\cal Q}(\nu)$, it follows from Definition~\ref{def:relax-sol} that (i) $Q_k\circ\zeta_k^{-1}=\nu$; (ii) $W_k$ consists of a sequence of independent Wiener processes on $(\Omega_{\infty},\Fx,Q_k)$; (iii) $\theta_k\in\mathbb{U}^{Q_k,\Fx}$. Moreover, for $i=1,\ldots,N$, {the state process $X_k^{\theta,i}$} is the strong solution of \eqref{eq:modeli} driven by $(\zeta_k,W_k,\theta_k)$. Then, it follows from \eqref{eq:LNRN}, \eqref{eq:minimizingse} and \eqref{eq:JNQk00} that, for all $k\geq1$,
\begin{align}
(\lambda_1\wedge\lambda_2)\Ex^{Q_k}\left[\int_0^T\{\left|\theta_k(t)\right|^2+|\theta_k'(t)|^2\}dt\right]\leq J_N(Q_k)\leq \alpha_N+\frac{1}{k}.
\end{align}
This implies that
\begin{align}\label{eq:unformesttheta}
\sup_{k\geq1}\Ex^{Q_k}\left[\int_0^T\{\left|\theta_k(t)\right|^2+|\theta_k'(t)|^2\}dt\right]\leq \frac{\alpha_N}{\lambda_1\wedge\lambda_2}.
\end{align}
Note that $\theta_k$ is a ${\cal H}_m^1$ (as a subset of ${\cal C}_m$)-valued random variable on $(\Omega_{\infty},\F_{\infty},Q_k)$ for $k\geq1$. Then, for any $\delta>0$, it follows from H\"older inequality that
\begin{align}\label{eq:thetaktsderi}
\sup_{k\geq1}\Ex^{Q_k}\left[{\sup_{\substack{|t-s|\leq\delta,\\0\leq s,t\leq T}}}\left|\theta_k(t)-\theta_k(s)\right|^2\right]\leq \sup_{k\geq1}\Ex^{Q_k}\left[{\sup_{\substack{|t-s|\leq\delta,\\0\leq s,t\leq T}}}\left|t-s\right|\int_0^T\left|\theta_k'(u)\right|^2du\right]\leq\frac{\alpha_N\delta}{\lambda_1\wedge\lambda_2}.
\end{align}
For any $\epsilon>0$ and $\delta>0$, using Chebychev's inequality, we arrive at
\begin{align}
&\sup_{k\geq1}{Q_k}\circ\theta_k^{-1}\left(\left\{h\in{\cal C}_m;\ {\sup_{\substack{|t-s|\leq\delta,\\0\leq s,t\leq T}}}\left|h(t)-h(s)\right|>\epsilon\right\}\right)=\sup_{k\geq1}Q_k\left({\sup_{\substack{|t-s|\leq\delta,\\0\leq s,t\leq T}}}\left|\theta_k(t)-\theta_k(s)\right|>\epsilon\right)\nonumber\\
&\qquad\qquad\leq\frac{1}{\epsilon^2}\Ex^{Q_k}\left[{\sup_{\substack{|t-s|\leq\delta,\\0\leq s,t\leq T}}}\left|\theta_k(t)-\theta_k(s)\right|^2\right]\leq\frac{\alpha_N\delta}{(\lambda_1\wedge\lambda_2)\epsilon^2}.
\end{align}
This implies that, for any $\epsilon>0$,
\begin{align}\label{eq:Azera-Ascoli1}
&\lim_{\delta\to0}\sup_{k\geq1}Q_k\circ\theta_k^{-1}\left(\left\{h\in{\cal C}_m;\ {\sup_{\substack{|t-s|\leq\delta,\\0\leq s,t\leq T}}}\left|h(t)-h(s)\right|>\epsilon\right\}\right)=0.
\end{align}
Using similar estimates as~\eqref{eq:thetaktsderi}, we obtain that, for all $s\in[0,T]$,
\begin{align}
\Ex^{Q_k}\left[\sup_{t\in[0,T]}\left|\theta_k(t)\right|^2\right]\leq 2\Ex^{Q_k}\left[\left|\theta_k(s)\right|^2\right]+2T\Ex^{Q_k}\left[\int_0^T|\theta_k'(u)|^2du\right],\quad\forall\ k\geq1.
\end{align}
Integrating both sides of the above equation w.r.t. $s$, it follows from Fubini's theorem that
\begin{align}
\Ex^{Q_k}\left[\sup_{t\in[0,T]}\left|\theta_k(t)\right|^2\right]\leq \frac{2}{T}\Ex^{Q_k}\left[\int_0^T\left|\theta_k(s)\right|^2ds\right]+2T\Ex^{Q_k}\left[\int_0^T|\theta_k'(u)|^2du\right],\quad\forall\ k\geq1.
\end{align}
Then, for any $M>0$, we obtain from \eqref{eq:unformesttheta} that
\begin{align}
&\sup_{k\geq1}Q_k\circ\theta_k^{-1}\left(\left\{h\in {\cal C}_m;\ \sup_{t\in[0,T]}\left|h(t)\right|>M\right\}\right)=\sup_{k\geq1}Q_k\left(\sup_{t\in[0,T]}\left|\theta_k(t)\right|>M\right)\nonumber\\
&\qquad\qquad\leq\frac{1}{M^2}\sup_{k\geq1}\Ex^{Q_k}\left[\sup_{t\in[0,T]}\left|\theta_k(t)\right|^2\right]\leq\frac{1}{M^2}\left(2T+\frac{2}{T}\right)\frac{\alpha_N}{\lambda_1\wedge\lambda_2}.
\label{eq:supthetat}
\end{align}
By virtue of Arzel\`a-Ascoli theorem (see, e.g. \cite{Simon1987}) together with  \eqref{eq:Azera-Ascoli1} and \eqref{eq:supthetat}, we obtain that $(Q_k\circ\theta_k^{-1})_{k=1}^{\infty}$, viewed as a sequence of probability measures in ${\cal P}({\cal C}_m)$, is tight.

Note that $\Omega_\infty\subset\hat{\Omega}_\infty$ where $\hat{\Omega}_\infty:=\Xi_K^{\N}\times {\cal C}_p^{\mathbb{N}}\times{\cal C}_m$. We next claim that the sequence of probability measures $(Q_k\circ{\cal X}_k^{-1})_{k=1}^{\infty}$ in ${\cal P}(\hat{\Omega}_\infty)$ is tight.
Since we have proven above that $(Q_k\circ\theta_k^{-1})_{k=1}^{\infty}\subset{\cal P}({\cal C}_m)$, it suffices to show that $(Q_k\circ\zeta_k^{-1})_{k=1}^{\infty}\subset{\cal P}(\Xi_K^{\N})$ and $(Q_k\circ W_k^{-1})_{k=1}^{\infty}\subset{\cal P}({\cal C}_p^{\N})$ are respectively tight. Note that $(Q_k)_{k=1}^{\infty}\subset{\cal Q}(\nu)$, then $Q_k\circ\zeta_k^{-1}=\nu$ for $k\geq1$. Hence, the tightness of $(Q_k\circ \zeta_k^{-1})_{k=1}^{\infty}$ follows from the fact that $(\Xi_K^{\N},d_1)$ is Polish. Similarly, it follows from (ii) of Definition~\ref{def:relax-sol} that the r.v.s $W_k$, for $k\geq1$, have the same distribution. Thus, $Q_1\circ W_1^{-1}=Q_k\circ{W}_k^{-1}$ for all $k\geq1$, and hence $(Q_k\circ W_k^{-1})_{k=1}^{\infty}\subset{\cal P}({\cal C}_p^{\N})$ is tight, because $({\cal C}_p^{\N},d_2)$ is Polish. By Prokhorov's theorem, there exists a $Q^*\in{\cal Q}(\nu)$ such that $Q_k\circ{\cal X}_k^{-1}$ converges to (up to a subsequence) $Q^*$ in the weak topology of probability measures. Using the Skorokhod's representation theorem, there exists a probability space $(\Omega^*,\F^*,\Px^*)$, $\Omega_{\infty}$-valued r.v.s ${\cal X}_k^*=(\zeta^*_k,W_k^*,\theta_k^*)$ with ${\cal X}_k^*\overset{d}{=}{\cal X}_k$, and  ${\cal X}^*=(\zeta^*,W^*,\theta^*)$ with $\Px^*\circ({\cal X}^*)^{-1}=Q^*$ such that, $\Px^*$-a.s., as $k\to\infty$,
\begin{align}\label{eq:Pstarconver}
\zeta_k^*&\to \zeta^*\ {\rm in}\ \Xi_K^{\N};\quad W_k^*\to W^*\ {\rm in}\ {\cal C}_p^{\N};\quad
\theta_k^*\to \theta^*\ {\rm in}\ {\cal C}_m.
\end{align}

Let ${\cal D}:=C_0^{\infty}(\R^m)$ be the space of test functions with its dual space given by ${\cal D}'$. Denote by $(\cdot,\cdot)$ the dual pair between ${\cal D}$ and ${\cal D}'$. We next prove that $\theta^*_k$ is ${\cal H}_m^1$-valued. For any $\phi\in{\cal D}$, we define the linear functional ${\cal T}_k(\phi):=(\theta_k^*,\phi')$ on ${\cal D}$, for $k\geq1$. Let $(\phi_l)_{l=1}^{\infty}\subset{\cal D}$ be dense in ${\cal L}_m^2$ and consider
\begin{align}
\left\|{\cal T}_k\right\|:=\sup_{l\geq1}\frac{|{\cal T}_k(\phi_l)|}{\left\|\phi_l\right\|_{{\cal L}_m^2}}=\sup_{l\geq1}\frac{|(\theta_k^*,\phi_l')|}{\left\|\phi_l\right\|_{{\cal L}_m^2}}.
\end{align}
Let $\Ex^*$ be the expectation operator under $\Px^*$. Then, for all $k\geq1$, it holds that
\begin{align}
\Ex^*\left[\sup_{l\geq1}\frac{\left|(\theta_k^*,\phi_l')\right|}{\left\|\phi_l\right\|_{{\cal L}_m^2}}\right]
=\Ex\left[\sup_{l\geq1}\frac{\left|(\theta_k',\phi_l)\right|}{\left\|\phi_l\right\|_{{\cal L}_m^2}}\right]=\Ex\left[\left\|\theta_k'\right\|_{{\cal L}_m^2}\right]<+\infty,
\end{align}
and hence $\|{\cal T}_k\|<+\infty$, $\Px^*$-a.s.. Then, by Hahn-Banach theorem, it holds $\Px^*$-a.s. that ${\cal T}_k$ can be extended to be a bounded linear functional on ${\cal L}_m^2$. Thus, Riesz representation theorem yields the existence of an ${\cal L}_m^2$-valued random variable $\hat{\theta}_k^*$ such that ${\cal T}_k(\phi)=(\hat{\theta}_k^*,\phi)$ for all $\phi\in {\cal L}_m^2$, $\Px^*$-a.s. In particular, ${\cal T}_k(\phi)=(\theta_k^*,\phi')=(\hat{\theta}_k^*,\phi)$ for any $\phi\in{\cal D}$. Hence, $\theta_k^*\in{\cal H}_m^1$, $\Px^*$-a.s. It follows that $\Px^*\circ({\cal X}_k^*)^{-1}\in{\cal P}(\Omega_{\infty})$. Moreover, we obtain from \eqref{eq:unformesttheta} that
\begin{align}
\sup_{k\geq1}\Ex^*\left[\left\|{\theta_k^*}'\right\|_{{\cal L}_m^2}\right]=\sup_{k\geq1}\Ex^*\left[\sup_{l\geq1}\frac{|(\theta_k^*,\phi_l')|}{\|\phi_l\|_{{\cal L}_m^2}}\right]
=\sup_{k\geq1}\Ex\left[\sup_{l\geq1}\frac{|(\theta_k,\phi_l')|}{\|\phi_l\|_{{\cal L}_m^2}}\right]=\sup_{k\geq1}\Ex\left[\left\|\theta_k'\right\|_{{\cal L}_m^2}\right]<\infty.
\end{align}
This implies that $({\theta_k^*}')_{k=1}^{\infty}$ is bounded in $L^2((0,T)\times\Omega^*;dt\otimes d\Px^*)$, and hence ${\theta_k^*}'$ (up to a subsequence) converges weakly to some element $\hat{\theta}^*\in L^2((0,T)\times\Omega^*;dt\otimes d\Px^*)$ as $k\to\infty$. Let $\phi\in{\cal D}$. By \eqref{eq:Pstarconver}, we have that, for all $H\in L^{\infty}(\Omega^*;\Px^*)$,
\begin{align}
\Ex^*\left[(\theta^*,\phi')H\right]=\lim_{k\to\infty}\Ex^*\left[(\theta_k^*,\phi')H\right]=-\lim_{k\to\infty}\Ex^*\left[({\theta_k^*}',\phi)H\right]
=-\Ex^*[(\hat{\theta}^*,\phi)H].
\end{align}
Then $(\theta^*,\phi')=-(\hat{\theta}^*,\phi)$, $\Px^*$-a.s. Therefore, using the separability of ${\cal D}$, it holds $\Px^*$-a.s. that $(\theta^*,\phi')=-(\hat{\theta}^*,\phi)$ for all $\phi\in{\cal D}$. This gives that ${\theta^*}'=\hat{\theta}^*$, $\Px^*$-a.s. Thus, $\Px^*\circ({\cal X}^*)^{-1}\in{\cal P}(\Omega_\infty)$.

For $k\geq1$ and $i\in\N$, let {$X_k^{*,i}$} be the strong solution of SDE~\eqref{eq:modeli} driven by $(\zeta_k^*,W_k^*,\theta_k^*)$. In other words, under $(\Omega^*,\F^*,\Px^*)$, {$(X_k^{*,i}(0),Y_k^{*,i}(0))=\zeta_k^{*,i}$}, and for $t\in(0,T]$,
{\begin{align}
\displaystyle dX_k^{*,i}(t)= f\left(t, \theta_k^*(t) ,X_k^{*,i}(t),\frac{1}{N}\sum_{j=1}^N\rho(X_k^{*,j}(t))\right)dt + \varepsilon^idW_k^{*,i}(t).\label{eq:modelistark}
\end{align}}
Moreover, for $i\in\N$, let {$X^{*,i}$} be the strong solution of SDE~\eqref{eq:modeli} driven by $(\zeta^*,W^*,\theta^*)$. Under $(\Omega^*,\F^*,\Px^*)$, {$(X^{*,i}(0),Y^{*,i}(0))=\zeta^{*,i}$}, and for $t\in(0,T]$,
\begin{align}
\displaystyle dX^{*,i}(t)= f\left(t, \theta^*(t),X^{*,i}(t),\frac{1}{N}\sum_{j=1}^N\rho(X^{*,j}(t))\right)dt + \varepsilon^idW^{*,i}(t).\label{eq:modelistar2}
\end{align}
By the assumption {\Afg}, we obtain that, for all $t\in[0,T]$, $\Px^*$-a.s.
\begin{align}
&\sup_{s\in[0,t]}\left|X_k^{*,i}(s)-X^{*,i}(s)\right|^2\leq C_{T}\Bigg\{\left|X_k^{*,i}(0)-X^{*,i}(0)\right|^2+\int_0^t\left|\theta_k^*(s)-\theta^*(s)\right|^2ds+\int_0^t\left|X_k^{*,i}(s)-X^{*,i}(s)\right|^2ds\nonumber\\
&\qquad\quad+\int_0^t\left|\frac{1}{N}\sum_{j=1}^N(\rho(X_k^{*,j}(s))-\rho(X^{*,j}(s)))\right|^2ds+|\varepsilon^i|\sup_{s\in[0,t]}\left|W_k^{*,i}(s)-W^{*,i}(s)\right|^2\Bigg\}.
\end{align}
Using the Lipschitz property of $\rho:\R^d\to\R^{q}$, we obtain from Jensen's inequality that
\begin{align}
&\left|\frac{1}{N}\sum_{j=1}^N(\rho(X_k^{*,j}(t))-\rho(X^{*,j}(t)))\right|^2\leq
[\rho]_{\rm Lip}^2\frac{1}{N}\sum_{j=1}^N\left|X_k^{*,j}(t)-X^{*,j}(t)\right|^2.
\end{align}
Define $|\tilde{x}|_N^2:=\frac{1}{N}\sum_{i=1}^N|x^i|^2$ for $\tilde{x}=(x^1,\ldots,x^N)\in(\R^{d})^N$. It holds that, $\Px^*$-a.s.
{\begin{align}
&\sup_{s\in[0,t]}\left|\tilde{X}_k^{*}(s)-\tilde{X}^{*}(s)\right|_N^2\leq e^{C_{T,N}}\left\{\left|\tilde{X}_k^{*}(0)-\tilde{X}^{*}(0)\right|_N^2+\int_0^t\left|\theta_k^*(s)-\theta^*(s)\right|^2ds+\sup_{s\in[0,t]}\left|\tilde{W}_k^{*}(s)-\tilde{W}^{*}(s)\right|_N^2\right\}.\nonumber
\end{align}}
Therefore, using {the convergence results from \eqref{eq:Pstarconver}}, we conclude that, as $k\to\infty$,
\begin{align}\label{eq:XNstarconver}
\sup_{s\in[0,t]}\left|\tilde{X}_k^{*}(s)-\tilde{X}^{*}(s)\right|_N^2\to 0,\qquad \Px^*\text{-a.s.}
\end{align}
From \eqref{eq:Pstarconver} and \eqref{eq:XNstarconver}, it follows that, $\Px^*$-a.s., as $k\to\infty$,
\begin{align}\label{eq:objecconN}
&\alpha\left|\tilde{X}_k^{*}(T)-\tilde{Y}_k^*(0)\right|_N^2+\beta\int_0^T\left|\tilde{X}_k^{*}(t)-\tilde{Y}_k^*(0)\right|_N^2dt+\lambda_1\int_0^T\left|\theta_k^*(t)\right|^2dt\nonumber\\
&\qquad\quad\too \alpha\left|\tilde{X}^{*}(T)-\tilde{Y}^*(0)\right|_N^2+\beta\int_0^T\left|\tilde{X}^{*}(t)-\tilde{Y}^*(0)\right|_N^2dt+\lambda_1\int_0^T\left|\theta^*(t)\right|^2dt.
\end{align}
It follows from properties of convex functionals and weak convergence (see, e.g. Theorem 1.4 in \cite{Figueiredo1989}) that
\begin{align}\label{eq:thetaderifatou}
\Ex^*\big[\|{\theta^{*}}'\|_{{\cal L}_m^2}^2\big]\leq \liminf_{k\to\infty}\Ex^*\big[\|{\theta_k^{*}}'\|_{{\cal L}_m^2}^2\big].
\end{align}
Recall the sampled objective functional given in \eqref{eq:control2Q}, and the square form of the loss function and regularizer given in \eqref{eq:LNRN}. Note that $Q_k=\Px^*\circ({\cal X}_k^*)^{-1}$, and $\Px^*\circ({\cal X}^*)^{-1}=Q^*$. Then, it follows from \eqref{eq:thetaderifatou} that $Q^*\in{\cal Q}(\nu)$. Moreover, by Fatou's lemma, it follows that
\begin{align}
J_N(Q^*)=J_N(\Px^*\circ({\cal X}^*)^{-1})\leq\liminf_{k\to\infty}J_N(\Px^*\circ({\cal X}_k^*)^{-1})=\liminf_{k\to\infty}J_N(Q_k).
\end{align}
We then deduce that $J_N(Q^*)\leq\alpha_N$ by using \eqref{eq:minimizingse}. Recall that $Q^*\in{\cal Q}(\nu)$ and hence $\alpha_N\leq J_N(Q^*)$. Therefore $J_N(Q^*)=\alpha_N$, i.e., $Q^*\in{\cal Q}(\nu)$ is the optimal relaxed solution of \eqref{eq:relaxed-controlQ}. This ends the proof.
\end{proof}

{The proofs of Lemma~\ref{lem:hatSequibounded-cont} and Lemma~\ref{lem:relativecomW2}, rely on the following technical lemma. Its proof is omitted because it is standard, and based on an application of the It\^o's formula and use of the Doob's maximal inequality (the proof can be provided upon request).}
\begin{lemma}\label{lem:estimateXN}
Let $\tilde{X}_N(t)=(X_N^1(t),\ldots,X_N^N(t))$ for $t\in[0,T]$ with $X_N^i(t)$ satisfying SDE~\eqref{eq:modelinonstark}. Let {\Afg} and {\Atheta} hold. Then, for all $p\geq1$, there exists a positive constant $C_p$ which is independent of $N$ such that
\begin{align}\label{eq:estimate-XN}
\Ex^{Q_N}\left[\sup_{t\in[0,T]}\left|\tilde{X}_N(t)\right|_N^{2p}\right]\leq C_p,\qquad \forall\ N\geq1.
\end{align}
\end{lemma}

\begin{proof}[Proof of Lemma~\ref{lem:hatSequibounded-cont}]
It follows from \eqref{eq:QxN} and Chebyshev's inequality that
\begin{align}
&\sup_{N\geq1}\Qx^N\left(\left\{\vartheta\in\hat{S};\ \sup_{t\in[0,T]}\int_E|e|^{2+\epsilon}\vartheta(t,de)> M\right\}\right)\leq\frac{1}{M}\sup_{N\geq1}\Ex^{Q_N}\left[\sup_{t\in[0,T]}\int_E|e|^{2+\epsilon}\mu^N(t,de)\right]\nonumber\\
&\qquad\qquad\leq\frac{1}{M}\sup_{N\geq1}\Ex^{Q_N}\left[\sup_{t\in[0,T]}\frac{1}{N}\sum_{i=1}^N\left|(\xi_N^i,X_N^{i}(t))\right|^{2+\epsilon}\right]\nonumber\\
&\qquad\qquad\leq\frac{C_{\epsilon}}{M}\Bigg\{\sup_{N\geq1}\Ex^{Q_N}\left[\frac{1}{N}\sum_{i=1}^N\left|\xi_N^i\right|^{2+\epsilon}\right]+\sup_{N\geq1}\Ex^{Q_N}\left[\sup_{t\in[0,T]}\left|\tilde{X}_N(t)\right|_{N}^{2+\epsilon}\right]\Bigg\}.
\end{align}
By the assumption {\Afg}-(i) and noting that $\zeta^i\in\Xi_K$ for all $i\geq1$, the limit \eqref{eq:QNM0} follows from  Lemma~\ref{lem:estimateXN}.
Using the representation of the empirical measure given by \eqref{eq:empiricalpmN}, it holds that, for $s,t\in[0,T]$,
{\begin{align}
{\cal W}_{E,2}^2\left(\mu^N(t),\mu^N(s)\right)\leq\frac{1}{N}\sum_{i=1}^N\left|X_N^{i}(t)-X_N^{i}(s)\right|^2.
\end{align}}
Then, it follows from Chebyshev's inequality that
\begin{align*}
&\sup_{N\geq1}\Qx^N\left(\left\{\vartheta\in\hat{S};\ \sup_{|t-s|\leq\delta}{\cal W}_{E,2}(\vartheta(t),\vartheta(s))>\varepsilon\right\}\right)\leq\frac{1}{\varepsilon^2}\sup_{N\geq1}\Ex^{Q_N}\left[\frac{1}{N}\sum_{i=1}^N\sup_{|t-s|\leq\delta}\left|X_N^{i}(t)-X_N^{i}(s)\right|^2\right].
\end{align*}
Using the assumption {\Afg} and Lemma~\ref{lem:estimateXN} with the assumption {\Atheta}, we deduce the existence of a positive constant $C$ which is independent of $N$ such that
{\begin{align}\label{eq:ZXiaverage}
\sup_{N\geq1}\Ex^{Q_N}\left[\frac{1}{N}\sum_{i=1}^N\sup_{|t-s|\leq\delta}\left|X_N^{i}(t)-X_N^{i}(s)\right|^2\right]&\leq C\delta\{1+\Upsilon(\delta)\},
\end{align}}
where, for $\delta>0$, we have defined
\begin{align}\label{eq:GammaN}
\Upsilon(\delta):=\sup_{N\geq1}\Ex^{Q_N}\left[\frac{1}{N}\sum_{i=1}^N\sup_{|t-s|\leq\delta}\left|W_N^{i}(t)-W_N^{i}(s)\right|^2\right].
\end{align}
Note that $(W_N^1,\ldots,W_N^N)$ are independent Wiener processes under $Q_N$. Note that, it holds that
\begin{align*}
\Upsilon(\delta)&=\sup_{N\geq1}\Ex^{Q_N}\left[\frac{1}{N}\sum_{i=1}^N\sup_{|t-s|\leq\delta}\left|W_N^{i}(t)-W_N^{i}(s)\right|^2\right]
=\sup_{N\geq1}\Ex^{Q_N}\left[\sup_{|t-s|\leq\delta}\left|W_N^1(t)-W_N^1(s)\right|^2\right]\nonumber\\
&=\Ex^{Q_1}\left[\sup_{|t-s|\leq\delta}\left|W_1^1(t)-W_1^1(s)\right|^2\right].
\end{align*}
Since $W_1^1$ is a Wiener process under $Q_1$, we have $\lim_{\delta\to0}\sup_{|t-s|\leq\delta}|W_1^1(t)-W_1^1(s)|^2=0$, $Q_1$-a.s.. On the other hand, the BDG inequality yields $\Ex^{Q_1}[\sup_{t\in[0,T]}|W_1^1(t)|^2]\leq C_T$ for some $C_T>0$ depending on $T$ only. Then, it follows from DCT that
$\Upsilon(\delta)\to0$ as $\delta\to0$. Hence, the limit \eqref{eq:QNM2} follows from \eqref{eq:ZXiaverage}.
\end{proof}

\begin{proof}[Proof of Lemma~\ref{lem:Wtouni}]
Define the mappings $\hat{I}_N,\hat{I}_{*}:\Xi_K^{\N}\rightarrow\Omega^0_\infty\times {\cal P}_2(E)$ as follows: for any $\hat{\zeta}\in\Xi_K^{\N}$,
	\begin{align}\label{eq:hatIN}
	\hat{I}_N(\hat{\zeta}):=(\hat{\zeta},I_N(\hat{\zeta})),\quad \hat{I}_*(\hat{\zeta}):=(\hat{\zeta},I_*(\hat{\zeta})).
	\end{align}
It follows from \eqref{eq:assumptionAnu3} in {\Anu} that $\nu\circ\hat{I}_N^{-1}\Rightarrow  \nu\circ\hat{I}_*^{-1}$ as $N\to\infty$. Observe that $Q_N\circ\hat{I}_N(\zeta_N)^{-1}=\{Q_N\circ\zeta^{-1}_N\}\circ\hat{I}_N^{-1}=\nu\circ\hat{I}_N^{-1}$, and $Q\circ\hat{I}_*(\zeta)^{-1}=\nu\circ\hat{I}_*^{-1}$. Then
\begin{align}\label{eq:weakcon000N}
	Q_N\circ\hat{I}_N(\zeta_N)^{-1}\Rightarrow Q\circ\hat{I}_*(\zeta)^{-1}=\nu\circ\hat{I}_*^{-1},\qquad N\to\infty.
\end{align}
Using the inequality ${\cal W}_{\Omega_{\infty},2}(Q_N\circ(\zeta_N,\theta_N)^{-1},Q\circ(\zeta,\theta)^{-1})\leq{\cal W}_{\Omega_{\infty},2}(Q_N,Q)$ and the assumption that $\lim_{N\to\infty}{\cal W}_{\Omega_{\infty},2}(Q_N,Q)=0$, we arrive at
\begin{align}\label{eq:weakcon111N}
	\lim_{N\to\infty}{\cal W}_{\Omega_{\infty},2}(Q_N\circ(\zeta_N,\theta_N)^{-1},Q\circ(\zeta,\theta)^{-1})=0.
\end{align}
Combining \eqref{eq:weakcon000N} and \eqref{eq:weakcon111N}, we obtain that $(Q_N\circ(\zeta_N,I_N(\zeta_N),\theta_N)^{-1})_{N=1}^{\infty}$ is tight.
	
We next prove that any convergent subsequence of $(Q_N\circ(\zeta_N,I_N(\zeta_N),\theta_N)^{-1})_{N=1}^{\infty}$ has the same weak limit. To start with, let $(N_k^i)_{k=1}^{\infty}$, $i=1,2$, be two subsequences of $\N$ such that $Q_{N_k^i}\circ(\zeta_{N_k^i},I_{N_k^i}(\zeta_{N_k^i}),\theta_{N_k^i})^{-1}\Rightarrow \Px^i\circ(\zeta_i,J_i,\theta_i)^{-1}$ as $k\to\infty$. Here, $(\zeta_i,J_i,\theta_i)$ is a $\Omega^0_\infty\times{\cal P}_2(E)\times{\cal H}^1_m$-valued random variable defined on some probability space $(\Omega^i,\F^i,\Px^i)$. By \eqref{eq:weakcon000N}, we have that $\Px^i\circ(\zeta_i,J_i)^{-1}=\nu\circ\hat{I}_*^{-1}$ for $i=1,2$. It then follows from \eqref{eq:hatIN} that, for $i=1,2$,
\begin{align}\label{eq:probab1nu}
	\Px^i\left(\left\{\omega\in\Omega^i;\ J_i=I_*(\zeta_i(\omega))\right\}\right)=\nu\left(\left\{\hat{\zeta}\in\Xi_K^{\N};\ \hat{I}_*(\hat{\zeta})=(x_1,x_2),\ x_2=I_*(x_1)\right\}\right)=1.
\end{align}
By applying the Gluing lemma (see Lemma 7.6 in \cite{Villani2003}), there exists a coupling $(J_1^*,J_2^*,\zeta^*,\theta^*)$ under some probability space $(\Omega^*,\F^*,\Px^*)$ such that $(\zeta^*,J_i^*,\theta^*)=(\zeta_i,J_i,\theta_i)$ in law for $i=1,2$. It then follows from \eqref{eq:probab1nu} that $J_1^*=J_2^*=I_*(\zeta^*)$, $\Px^*$-a.s. This yields that $\Px^{1}\circ(\zeta_1,J_1,\theta_1)^{-1}=\Px^{2}\circ(\zeta_2,J_2,\theta_2)^{-1}$, and hence every convergent subsequence of $(Q_N\circ(\zeta_N,I_N(\zeta_N),\theta_N)^{-1})_{N=1}^{\infty}$ admits the same weak limit. Moreover, for $Q\in{\cal Q}(\nu)$, the assumption {\Anu} together with Definition~\ref{def:relax-sol}-(i) yields
\begin{align*}
	&Q\left(\left\{\omega\in\Omega_{\infty};\ \lim_{N\to\infty}{\cal W}_{E,2}(I_N(\zeta(\omega)),I_*(\zeta(\omega)))=0\right\}\right)=\nu\left(\left\{\hat{\zeta}\in\Xi_K^{\N};\ \lim_{N\to\infty}{\cal W}_{E,2}(I_N(\hat{\zeta}),I_*(\hat{\zeta}))=0\right\}\right)=1,
\end{align*}
and hence $(\zeta,I_N(\zeta),\theta)\to(\zeta,I_*(\zeta),\theta)$, $Q$-a.s. This concludes the proof that $Q\circ(\zeta,I_N(\zeta),\theta)^{-1}\Rightarrow Q\circ(\zeta,I_*(\zeta),\theta)^{-1}$ as $N\to\infty$.
	Moreover, define a specific sequence $(\hat{Q}_N)_{N=1}^{\infty}\subset{\cal Q}(\nu)$ as $\hat{Q}_{2l-1}:=Q_{2l-1}$ and $\hat{Q}_{2l}=Q$ for all $l\in\N$. We then have that $\hat{Q}_N\circ(\zeta_N,I_N(\zeta_N),\theta_N)^{-1}\Rightarrow{Q}\circ(\zeta,I_*(\zeta),\theta)^{-1}$. This proves the weak convergence result in the lemma. 
\end{proof}

\begin{proof}[Proof of Lemma~\ref{lem:limitQstar}]
The uniqueness of a solution to the FPK equation~\eqref{eq:mustar1} in the trajectory sense follows from Proposition \ref{prop:limitP1}. We next show the existence. For given $(I_*,\theta)\in{\cal P}_2(E)\times{\cal C}_m$, consider the weak solution of the following parameterized SDE defined on a filtered probability space $(\hat{\Omega},\hat{\F},\hat{\Fx}=(\hat{\F}_t)_{t\in[0,T]},\hat{\Px})$ which supports a $p$-dimensional Brownian motion $\hat{W}=(\hat{W}(t))_{t\in[0,T]}$, and r.v.s {$(Y(0),X(0))\in\hat{\F}_0$:
\begin{align}\label{eq:XxiIstartheta} dX^{\xi,I_*,\theta}(t)=f\big(t,\theta(t),X^{\xi,I_*,\theta}(t),\hat{\Ex}[\rho(X^{\xi,I_*,\theta}(t))]\big)dt+\varepsilon d\hat{W}(t),
\end{align}
and $(\varepsilon,Y(0),X(0))$} admits the law $I_*$. It can then be verified that {$\mu(t):=\hat{\Px}\circ((\varepsilon,Y(0)),X^{\xi,I_*,\theta}(t))^{-1}$}, for $t\in[0,T]$, satisfies Eq.~\eqref{eq:mustar1}.

Moreover, by Lemma \ref{lem:Wtouni}, $Q_N\circ(\mu^N(0),\theta_N)^{-1}\Rightarrow Q\circ(I_*,\theta)^{-1}$ as $N\to\infty$. It follows from Theorem \ref{thm:limitP} that $\Qx^N:=Q_N\circ(\mu^N(0),\theta_N,\mu^N)^{-1}\to\hat{\Px}\circ(\hat{\mu}_0,\hat{\theta},\hat{\mu})^{-1}$ in ${\cal P}_2({\cal P}_2(E)\times{\cal C}_m\times\hat{S})$ for some probability space $(\hat{\Omega},\hat{\Fx},\hat{\Px})$, where $\hat{\Px}$-a.s., $\hat{\mu}$ is the unique solution of FPK equation \eqref{eq:mustar} with initial condition $\hat{\mu}(0)=\hat{\mu}_0$. Note that $(\hat{\mu}_0,\hat{\theta})$ has the law given by $Q\circ(I_*,\theta)^{-1}$. Then, from Gluing lemma, there exists a coupling $(\bar{\mu}_0,\bar{\theta},\bar{\mu}^1,\bar{\mu}^2)$ on some probability space $(\bar{\Omega},\bar{\F},\bar{\Px})$ such that $\bar{\Px}\circ(\bar{\mu}_0,\bar{\theta},\bar{\mu}^1)^{-1}=\hat{\Px}\circ(\hat{\mu}_0,\hat{\theta},\hat{\mu})^{-1}$ and $\bar{\Px}\circ(\bar{\mu}_0,\bar{\theta},\bar{\mu}^2)^{-1}=Q\circ(I_*,\theta,\mu_*)^{-1}$. Recall here that, $Q$-a.s., $\mu_*$ solves FPK equation \eqref{eq:mustar} with initial condition $\mu_*(0)=I_*$. Using a similar proof to that of Lemma \ref{lem:unilimit}, it follows that $(\hat{\mu}_0,\hat{\theta},\hat{\mu})$ and $(I_*,\theta,\mu_*)$ are identical in law. Taking $Q_N=Q$ for all $N\geq1$, we deduce that $\hat{\Qx}^*=Q\circ(I_*,\theta,\mu_*)^{-1}$.
\end{proof}

\noindent{\bf Proofs for the case $\lambda_2 = 0$ in \eqref{eq:squareLR}.}\quad We sketch the proofs for the case $\lambda_2 = 0$ in \eqref{eq:squareLR}. {We impose the assumption that the forcing function $f$ in Section~\ref{sec:concludes}}
is of the form specified by  \eqref{eq:dynamf}, both in the case of finite and infinite sample size. If the sample size is finite,  we can use the HJB equation to establish the existence of minimizers in the feedback form. Then, using the standard verification argument, we can show that the value function of relaxed controls coincides with the one of strong controls. Moreover, given the existence of optimal (feedback) controls, using stochastic maximum principle, we can study the convergence of minimizers as $N$ tends to infinity. For each fixed sample size $N$, we apply the stochastic maximum principle to characterize the optimal control $\theta_N=(\theta_N^{(1)}(t),\theta_N^{(2)}(t))_{t\in[0,T]}$. To adopt this framework, we take the canonical space in \eqref{eq:OmegaNFN} as $\Omega_\infty=\Xi_K^{\mathbb{N}}\times {\cal C}_{p}^{\mathbb{N}}\times L^1([0,T];\R^{d}\times\R^l)$ and use Definition \ref{def:relax-sol} on relaxed control without the condition (iii). For $i=1,\ldots,N$, consider the adjoint process $(p_N^i,q_N^i)=(p_N^i(t),q_N^i(t))_{t\in[0,T]}\in\R^d\times\R^{d\times Nd}$ which satisfies the following linear BSDE:
\begin{align*}
	dp^i_N(t)&=-\left[\theta^{(1)}_N(t)\nabla_xS(i,t)p^i_N(t)+\frac{\theta^{(1)}_N(t)}N\sum_{j=1}^N\nabla_\eta S(i,t)\nabla_x\rho\big(X^j(t)\big)p^j_N(t)+\partial_iR_N(t)\right]dt\nonumber\\
	&\quad+q^i_N(t)dW(t),\\
	p^i_N(T)&=\partial_iL_N(T)=\frac\alpha N(X^i(T)-Y^i(0)),
\end{align*}
where, for $(t,i)\in[0,T]\times\{1,\ldots,N\}$,  $W(t)^\top=(W^1_N(t)^\top,W^2(t)^\top,\ldots,W^N(t)^\top)$, $\partial_iR_N(t):=\frac\beta N(X^i(t)-Y^i(0))$, $S(i,t):=S(t,X^i(t),\sum_{j=1}^N\rho(X^j(t)))$ and $\nabla_{\ell}S(i,t):=\nabla_{\ell}S(t,X^i(t),\sum_{j=1}^N\rho(X^j(t)))$ for $\ell\in\{x,\eta\}$.
For $t\in[0,T]$, denote by
\begin{align*}
	&p_N(t)^\top:=(p^1_N(t)^\top,\ldots,p^N_N(t)^\top),\quad q_N(t)^\top:=(q^1_N(t)^\top,\ldots,q^N_N(t)^\top),\notag\\
	&X(t)^\top:=(X^1(t)^\top,\ldots,X^N(t)^\top),\quad Y(0)^\top:=(Y^1(0)^\top,\ldots,Y^N(0)^\top).
\end{align*}
Note that $p_N(t),X(t),Y(0)\in\R^{Nd\times1}$. Let $A^N(t)=(A^N_{ij}(t))_{N\times N}$ whose $(i,j)$-th entry as $d\times d$-submatrix given by
\begin{align*}
	A^N_{ij}(t):=\left\{\begin{aligned}
		&\theta^{(1)}_N(t)\nabla_xS(i,t),\quad i=j,\\
		&\frac{\theta^{(1)}_N(t)}N\nabla_\eta S(i,t)\nabla_x\rho\big(X^j(t)\big),\quad i\neq j.
	\end{aligned}\right.
\end{align*}
Hence, for $A^N(t)=(a_{kl}^{N}(t))_{Nd\times Nd}$, we have from the compactness of $\Theta$ and Assumption {\Afg} that $|a_{kl}^{N}(t)|\leq C$ for $k=l$; while $|a_{kl}^{N}(t)|\leq C/(Nd)$ for $k\neq l$, where $C>0$ is a constant which depends on $T$ only. Moreover, it holds that
\begin{align}\label{expression-pt}
	p_N(t)&=\frac\alpha N\Ex_t\left[\Phi^{N}_{t,T}(A^N)(X(T)-Y(0))\right]
	-\frac\beta N\Ex_t\left[\int_t^T\Phi^{N}_{t,s}(A^N)(X(s)-Y(0))ds\right],
\end{align}
where $\Ex_t[\cdot]:=\Ex[\cdot|\F_t]$ for $t\in[0,T]$. Here, $\Phi^{N}_{t,T}(A^N)\in\R^{Nd\times Nd}$ satisfies that, a.s.
\begin{align}\label{def-Phi}
  d\Phi^N_{t,s}(A^N)=\Phi^N_{t,s}(A^N)A^N(s)ds,\quad s\in[t,T],\quad \Phi^N_{t,t}(A^N)=I_{Nd\times Nd}.
\end{align}
By the uniform estimate on $a_{kl}^N=(a_{kl}^N(t))_{t\in[0,T]}$ above, for $t\geq0$ fixed, we also have that
\begin{align}\label{eq:diemphi}
\left|(\Phi^{N}_{t+s,u}(A^N))_{kl}\right|\leq C,\quad \text{if}~k=l;\quad \left|(\Phi^{N}_{t+s,u}(A^N))_{kl}\right|\leq C/(Nd),\quad \text{if}~k\neq l,
\end{align}
for all $s\in[0,T-t]$ and $u\in[t+s,T]$. Here, $C>0$ is a constant depending on $T$ only. The optimality of the control $\theta_N=(\theta_N(t))_{t\in[0,T]}=(\theta_N^{(1)}(t),\theta_N^{(2)}(t))_{t\in[0,T]}$ yields that
\begin{align}\label{eq:porject}
	\theta_N(t)=P_\Theta\left(\tilde\theta_N(t)\right),\quad t\in[0,T],
\end{align}
where $P_\Theta$ is the projection on $\Theta$, and
\begin{align}\label{eq:tildethetaN}
	\tilde\theta_N(t)&=\arg\min_{\theta\in\Theta}\sum_{i=1}^N\left\langle\theta^{(1)}S(i,t)+\theta^{(2)},p_N^i(t)\right\rangle+{\rm tr}\left(\theta^{(1)}\left(\theta^{(1)}\right)^\top\right)+\left\langle\theta^{(2)},\theta^{(2)}\right\rangle.
\end{align}
Note that the functional on RHS of \eqref{eq:tildethetaN} is  quadratic in the control variable, and hence admits a unique minimizer:
\begin{align}\label{tilde-theta}
	\tilde\theta^{(1)}_N(t)&=-\frac12\sum_{i=1}^Np_N^i(t)S(i,t)^\top,\quad \tilde\theta^{(2)}_N(t)=-\frac12\sum_{i=1}^Np_N^i(t),\quad t\in[0,T].
\end{align}
Let $t\geq0$ and $\delta>0$. It follows from \eqref{eq:porject} that $\Ex[\sup_{0\leq s\leq\delta}|\theta_N(t+s)-\theta_N(t)|]\leq\Ex[\sup_{0\leq s\leq\delta}|\tilde\theta_N(t+s)-\tilde\theta_N(t)|]$. On the other hand, we deduce from \eqref{tilde-theta} that, for all $s\in[0,\delta]$,
\begin{align*}
  \left|\tilde\theta^{(1)}_N(t+s)-\tilde\theta^{(1)}_N(t)\right|\leq\frac{\|S\|_{\infty}}{2}\left|\sum_{i=1}^N(p_N^i(t+s)-p_N^i(t))\right|+\frac12\left|\sum_{i=1}^Np_N^i(t)(S(i,t+s)-S(i,t))^\top\right|,
\end{align*}
and $|\tilde\theta^{(2)}_N(t+s)-\tilde\theta^{(2)}_N(t)|\leq\frac12\sum_{i=1}^N|p_N^i(t+s)-p_N^i(t)|$. Below, let $C>0$ be a generic constant which depends on $T,d$ only, but it may be different from line to line. Using the condition satisfied by $S(t,x,\eta)$ in \eqref{eq:dynamf}, the probabilistic representation \eqref{expression-pt} of $p_N(t)$ and the estimate \eqref{eq:diemphi}, we obtain that, for all $s\in[0,\delta]$,
\begin{align}\label{differ-p-N}
 \Ex_t\left[\left|p_N(t+s)-p_N(t)\right|^2\right]\leq Cs\Ex_t\left[\max_{1\leq i\leq N}\sup_{v\in[t,T]}\left|X^i(v)-Y^i(0)\right|^2\right]
 \leq C\delta \left[\max_{1\leq i\leq N}|X^i(t)|^2+1\right],
\end{align}
and meanwhile, it holds that
\begin{align}\label{conclude-1}
  \Ex_t\left[\sup_{s\in[0,\delta]}\left|\sum_{i=1}^Np_N^i(t)(S(i,t+s)-S(i,t))^\top\right|^2\right]\leq C\delta\left[\max_{1\leq i\leq N}|X^i(t)|^2+1\right].
\end{align}
Thus, in light of \eqref{differ-p-N} and \eqref{conclude-1}, we arrive at
\begin{align}\label{L-tight}
  \Ex\left[\sup_{s\in[0,\delta]}\left|\tilde\theta_N(t+s)-\tilde\theta_N(t)\right|^2\right]\leq C\delta\Ex\left[\max_{1\leq i\leq N}|X^i(t)|^2+1\right].
\end{align}
Hence, from \eqref{L-tight} and the moment estimate of $X^i(t)$ for $i=1,\ldots,N$, it deduces that
\begin{align*}
  \Ex\left[\sup_{s\in[0,\delta]}\int_0^{T-s}\left|\tilde\theta_N(t+s)-\tilde\theta_N(t)\right|dt\right]
  \leq\int_0^T\Ex\left[\sup_{s\in[0,\delta]}\left|\tilde\theta_N(t+s)-\tilde\theta_N(t)\right|\right]dt\leq C\sqrt{\delta}.
\end{align*}
This yields from Chebyshev's inequality that, for any $\varepsilon>0$,
\begin{align}\label{eq:lamdachebi0}
  &\lim_{\delta\to0}\sup_{N\geq1}\mathbb P\circ\tilde\theta_N^{-1}\left(\left\{h\in L^1\left([0,T];\Theta\right);~{\sup_{s\in[0,\delta]}}\int_0^{T-s}\left|h(t+s)-h(t)\right|dt>\varepsilon\right\}\right)=0.
\end{align}
Note that $\tilde\theta_N$ is $\Theta$-valued, and $\Theta$ is compact. Then, Lemma A.2 in \cite{Barbu18} yields the tightness of $(Q_N)_{N\geq1}$ with $Q_N$ being the relaxed control associated with $\tilde\theta_N=(\tilde\theta_N(t))_{t\in[0,T]}$. Therefore, we have shown the compactness of sequence of minimizers for $J_N(Q)$ over all $Q\in{\cal Q}(\nu)$. Note also that $J_N\overset{\Gamma}{\to}J$, as $N\to\infty$ {(c.f. Remark~\ref{eq:lambda2eqn0})}, and thus the convergent subsequence of minimizers $(Q_N)_{N\geq1}$ for the objective functional $J_N$ converges to the minimizer $Q$ of the limit objective functional $J$, as $N\to\infty$.


\begin{thebibliography}{}

\bibitem[{Ahmed and Charalambous(2013)}]{AhmedCharalambous13} Ahmed, N.U., and  C.D. Charalambous (2013):  Stochastic minimum principle for partially observed systems subject to continuous and jump diffusion processes and driven by relaxed controls. {\it SIAM J. Control Optim.} 51, 3235-3257.


\bibitem[{Aliprantis and Border(2006)}]{AliprantisBorder2006} Aliprantis, C.D., and K. Border (2006): {\it Infinite Dimensional Analysis: A Hitchhiker's Guide.} Springer-Verlag, New York.

\bibitem[{Annunziato and Borz\`i(2010)}]{AnnunziatoBorzi2010} Annunziato, M., and A. Borz\`i (2010): Optimal control of probability density functions
of stochastic processes. {\it Math. Model. Anal.} 15, 393-407.


\bibitem[{Ba and Frey(2013)}]{Ba2013} Ba, J., and B. Frey (2013): Adaptive dropout for training deep neural networks. {\it Advances in Neural Information Processing Systems (NIPS)} 26, 1-9.

\bibitem[{Bahlali et al.(2007)}]{Bahlalietal2007} Bahlali, S., B. Djehiche, and  B. Mezerdi (2007): The relaxed stochastic maximum principle in singular optimal control of diffusions. {\it SIAM J. Control Optim.} 46, 427-444.

\bibitem[{Barbu et al.(2018)}]{Barbu18} Barbu, V., M. R\"{o}ckner, D. Zhang (2018): Optimal bilinear control of nonlinear stochastic Schr\"{o}dinger equations driven by linear multiplicative noise. {\it Ann. Probab.} 46, 1957-1999.

\bibitem[{Bencheikh and Jourdain(2021)}]{BencheikhJourdain21} Bencheikh, O., and B. Jourdain (2022): Weak and strong error analysis for mean-field rank-based particle approximations of one dimensional viscous scalar conservation laws. Forthcoming in {\it Ann. Appl. Probab.}


\bibitem[{Carmona and Delarue(2018)}]{CarmonaDelarue18} Carmona, R., and F. Delarue (2018): {\it Probabilistic Theory of Mean Field Games with Applications I-II}. Series: Probability Theory and Stochastic Modelling, Springer-Verlag, New York.

\bibitem[{Carmona et al.(2016)}]{CarmonaDelLa2016} Carmona, R., F. Delarue, and D. Lacker (2016):  Mean field games with common noise. {\it Ann. Probab.}
44, 3740-3803.

\bibitem[{Chen et al.(2018)}]{chenetal2018} Chen, R., Y. Rubanova, J. Bettencourt, and D. Duvenaud (2018): Neural ordinary differential equations. {\it Proceedings of 32nd Conference on Neural Information Processing Systems (NeurIPS)} 31, 6571-6583.

\bibitem[{Chizat and Bach(2018)}]{Chizat} Chizat, L. and F. Bach (2018): On the global convergence of gradient descent for over-parameterized models using optimal transport. {\it Proceedings of 32nd Conference on Neural Information Processing Systems (NeurIPS)} 31, 3040-3050.

\bibitem[{Cuchiero et al.(2020)}]{Larsson} Cuchiero, C., M. Larsson, and J. Teichmann (2020): Deep neural networks, generic universal interpolation, and controlled ODEs. {\it SIAM J. Math. Data Sci.} 2, 901-919.

\bibitem[{Dal Maso(1993)}]{Dalmaso1993} Dal Maso, G. (1993): {\it An Introduction to $\Gamma$-Convergence}. Springer-Verlag, New York.

\bibitem[{De Figueiredo(1991)}]{Figueiredo1989} De Figueiredo, D.G. (1991): Lectures on the Ekeland variational principle with applications and detours. {\it Acta. Appl. Math.} 24, 195-196.

\bibitem[{Dieng et al.(2018)}]{Blei} Dieng, A.B., R. Ranganath, J. Altosaar, and D.M. Blei (2018): Noisin: unbiased regularization for recurrent neural networks. {\it Proceedings of the 35th International Conference on Machine Learning (ICML-18)} 80, 1252-1261.

\bibitem[{Du et al.(2019)}]{Du19} Du, S.S., J. Lee, H. Li, L.~Wang, and X. Zhai (2019): Gradient descent finds global minima of deep neural networks. {\it Proceeding of the 36th International Conference on Machine Learning (ICML-19)} 97, 3003-3048.

\bibitem[{E(2017)}]{E17} E, Weinan (2017): A proposal on machine learning via dynamical systems. {\it Comm. Math. Stats.} 5, 1-11.

\bibitem[{E et al.(2018)}]{EHanLi18} E, Weinan, J.Q. Han, and Q.X. Li (2018): A mean-field optimal control formulation of deep learning. {\it Res. Math. Sci.} 6:10, 1-41.

\bibitem[{El Karoui et al.(1987)}]{ElKaroui87} El Karoui, N.,  D.H. Nguyen, and M. Jeanblanc-Picqu\`e (1987): Compactification methods in the control of degenerate diffusions: existence of an optimal control. {\it Stochastics} 20, 169-219.

\bibitem[{Evans(2010)}]{Evans2010} Evans, L.C. (2010): {\it Partial Differential Equations, 2nd Ed.}. AMS, Providence.


\bibitem[{Jabir et al.(2019)}]{Jabir} Jabir, J.F., D. Siska, and L. Szpruch (2019): Mean-field neural ODEs via relaxed optimal control. 	arXiv:1912.05475, available at \url{https://arxiv.org/abs/1912.05475}


\bibitem[{Haber and Ruthotto(2018)}]{HaberRuthotto2018} Haber, E., and L. Ruthotto (2018): Stable architectures for deep neural networks. {\it Inverse Problems} 34, 014004.

\bibitem[{Hasan and Roy-Chowdhury(2015)}]{HasanRoy-Chowdhury15} Hasan, M., and A.K. Roy-Chowdhury (2015): A continuous learning framework for activity recognition using deep hybrid feature models. {\it IEEE Trans. Multimedia} 17, 1909-1922.

\bibitem[{Haussmann and Lepeltier(1990)}]{HaussmannLepeltier90} Haussmann, U.G., and J.P. Lepeltier (1990): On the existence of optimal controls. {\it SIAM J. Control Optim.} 28, 851-902.

\bibitem[{Haykin(2009)}]{Haykin2009} Haykin, S. (2009): {\it Neural Network and Learning Machines, 3rd Ed..} Pearson Education Inc, New York.

\bibitem[{He et al.(2016)}]{HeRenSun2016} He, K., S.~Ren, and J.~Sun (2016): Deep residual learning for image recognition.  {\it 2016 IEEE Conference on Computer Vision and Pattern Recognition (CVPR)}, 770-778.

\bibitem[{He et al.(2019)}]{RakinHeFan2018} He, Z.H., A.S. Rakin, and D.L. Fan (2019): Parametric noise injection: Trainable randomness to improve deep neural network robustness against adversarial attack. {\it Proceedings-2019 IEEE/CVF Conference on Computer Vision and Pattern Recognition (CVPR)}, 588-597.

\bibitem[{Hinterm\"uller et al.(2013)}]{HinMarSpar13} Hinterm\"uller, M., D. Marahrens, P.A. Markowich, and C. Sparber (2013): Optimal
bilinear control of Gross-Pitaevskii equations. {\it SIAM J. Control Optim.} 51, 2509-2543.

\bibitem[{Ioffe and Szegedy(2015)}]{IoffeSzegedy2015} Ioffe, S., and C. Szegedy (2015): Batch normalization: accelerating deep network training by reducing internal covariate shift. {\it ICML'15: Proceedings of the 32nd International Conference on International Conference on Machine Learning} 37, 448-456.

\bibitem[{Kingma et al.(2015)}]{Kingma} Kingma, D.P., T. Salimans, and M. Welling (2015): Variational dropout and the local reparameterization trick. {\it  Advances in Neural Information Processing Systems} 28, 2575-2583.

\bibitem[{Lacker(2015)}]{Lacker15} Lacker, D. (2015): Mean field games via controlled martingale problems: Existence of Markovian equilibria. {\it Stochastic Process. Appl.} 125, 2856-2894.

\bibitem[{Lacker(2016)}]{Lacker16} Lacker, D. (2016): A general characterization of the mean field limit for stochastic differential games. {\it Probab. Theory Related Fields} 165, 581-648.

\bibitem[{Lu et al.(2018)}]{Lu} Lu, Y., A. Zhong, Q. Li, and B. Dong (2018): Beyond finite layer neural networks: Bridging deep architectures and numerical differential equations. {\it Proceedings of the 35 th International Conference on Machine Learning (PMLR)} 80, 5181-5190.


\bibitem[{Lu\c{c}on and Stannat(2014)}]{LuconStannat14} Lu\c{c}on, E., and W. Stannat (2014): Mean field limit for disordered diffusions with singular interactions. {\it Ann. Appl. Probab.} 24, 1946-1993.

\bibitem[{Mallat(2016)}]{Mallat16} Mallat, S. (2016): Understanding deep convolutional neural networks. {\it Philos. Trans. Roy. Soc. A} 374, 20150203.

\bibitem[{Manita et al.(2015)}]{Manitaetal2015} Manita, O.A., M.S. Romanov, and S.V. Shaposhnikov (2015): On uniqueness of solutions to nonlinear Fokker-Planck-Kolmogorov equations. {\it Nonlinear Anal.} 128, 199-226.

\bibitem[{Mei et al.(2018)}]{Montanari} Mei, S., A. Montanari, and P. Nguyen (2018): A mean field view of the landscape of two-layer neural networks. {\it PNAS}, 115, E7665-E7671.

\bibitem[{Motte and Pham(2022)}]{MottePhamAAP2021} Motte, M., and H. Pham (2022): Mean-field Markov decision processes with common noise and open-loop
controls. {\it Ann. Appl. Probab.} 32, 1421-1458.

\bibitem[{Noh et al.(2017)}]{NIPSStochAct} Noh, H., T. You,  J. Mun, and B. Han (2017): Regularizing deep neural networks by noise: Its interpretation and optimization. {\it NIPS'17: Proceedings of the 31st International Conference on Neural Information Processing Systems} 5115-5124.

\bibitem[{Oberman and Calder(2018)}]{Oberman2018} Oberman, A.M., and J. Calder (2018): Lipschitz regularized deep neural networks converge and generalize. 	ArXiv:1808.09540, available at \url{https://arxiv.org/abs/1808.09540v2}

\bibitem[{Rotskoff and Vanden-Eijnden(2019)}]{Rotskoff2018} Rotskoff, G.M., and E. Vanden-Eijnden (2019): Trainability and accuracy of neural networks: An interacting particle system approach. ArXiv:1805.00915, available at \url{https://arxiv.org/abs/1805.00915v3}


\bibitem[{Simon(1987)}]{Simon1987} Simon, J. (1987): Compact sets in the space $L^p(0,T;B)$. {\it Ann. Mat. Pura Appl.} 146, 65-96.



\bibitem[Sirignano and Spiliopoulos(2020)]{SS1} Sirignano, J. and K. Spiliopoulos (2021): Mean field analysis of neural networks: A law of large numbers. {\it SIAM J. Appl. Math.} 80, 725-752.

\bibitem[Sirignano and Spiliopoulos(2022)]{SS2} Sirignano, J., and K. Spiliopoulos (2022): Mean field analysis of deep neural networks. {\it Math. Oper. Res.} 47, 120-152.


\bibitem[{Srivastava et al.(2014)}]{SrivastavaMLR} Srivastava, N., G.E. Hinton, A. Krizhevsky, I. Sutskever, and R. Salakhutdinov (2014): Dropout: a simple way to prevent neural networks from overfitting. {\it J. Mach. Learn. Res.} 15, 1929-1958.


\bibitem[{Thorpe and van Gennip(2020)}]{ThGe2018} Thorpe, M., and Y. van Gennip (2020): Deep limit of residual neural networks. ArXiv:1810.11741, available at \url{https://arxiv.org/abs/1810.11741}

\bibitem[{Villani(2003)}]{Villani2003} Villani, C. (2003): {\it Topics in Optimal Transportation.} Graduate Studies in Mathematics, Volume 58, AMS.

\bibitem[{Villani(2009)}]{Villani2009} Villani, C. (2009): {\it Optimal Transport: Old and New Part.} Springer-Verlag, New York.

\bibitem[{Wan et al.(2013)}]{LecunICML} Wan, L., M. Zeiler, S. Zhang, Y. LeCun, and R. Fergus (2013): Regularization of neural networks using dropconnect. {\it Proceedings of the 30th International Conference on Machine Learning (PMLR)} 28, 1058-1066.


\end{thebibliography}
\end{document}